\documentclass[12pt]{elsarticle}
\usepackage{mystyleElsArticle}
\usepackage{mathrsfs}
\usepackage{booktabs}

\usepackage[margin=1.00 in]{geometry}    

\renewcommand\d{\mathrm{d}}

\newcommand{\veps}{\varepsilon}
\def\epsilon{\varepsilon}
\def\bx{\textbf{x}}
\def\bc{\textbf{c}}
\newcommand{\secref}[1]{\S\ref{#1}}

\newtheorem{example}[theorem]{Example}

\makeatletter
\def\ps@pprintTitle{%
	\let\@oddhead\@empty
	\let\@evenhead\@empty
	\let\@oddfoot\@empty
	\let\@evenfoot\@oddfoot
}
\makeatother

\begin{document}
	
	\begin{frontmatter}
		\title{ A multiscale finite element method for the Schr\"{o}dinger equation with multiscale potentials}

		\author[SoochowUniv]{Jingrun Chen}
		\ead{jingrunchen@suda.edu.cn}
		\author[hku]{Dingjiong Ma}
		\ead{martin35@hku.hk}
		\author[hku]{Zhiwen Zhang\corref{cor1}}
		\ead{zhangzw@hku.hk}
		
		\address[SoochowUniv]{Mathematical Center for Interdisciplinary Research and School of Mathematical Sciences, Soochow University, Suzhou, China.}
		\address[hku]{Department of Mathematics, The University of Hong Kong, Pokfulam Road, Hong Kong SAR, China.}
		\cortext[cor1]{Corresponding author}		
		
		\begin{abstract}
			\noindent	
			In recent years, an increasing attention has been paid to quantum heterostructures with tailored functionalities,
			such as heterojunctions and quantum matematerials, in which quantum dynamics of electrons can be described by the
			Schr\"{o}dinger equation with multiscale potentials. The model, however, cannot be solved by asymptoics-based approaches
			where an additive form of different scales in the potential term is required to construct the prescribed
			approximate solutions. In this paper, we propose a multiscale finite element method to solve this problem in the
			semiclassical regime. The localized multiscale basis are constructed using sparse compression of the Hamiltonian
			operator, and thus are ``blind" to the specific form of the potential term. After an one-shot eigendecomposition,
			we solve the resulting system of ordinary differential equations explicitly for the time evolution.
			In our approach, the spatial mesh size is $ H=\mathcal{O}(\epsilon) $ where $\epsilon$ is the semiclassical parameter and
			the time stepsize $ k$ is independent of $\epsilon$.
			Numerical examples in one dimension with a periodic potential, a multiplicative two-scale potential, and a layered potential,
			and in two dimension with an additive two-scale potential and a checkboard potential are tested to demonstrate the robustness
			and efficiency of the proposed method. Moreover, first-order and second-order rates of convergence are observed in $H^1$ and $L^2$ norms, respectively.
			
			
		\end{abstract}
		
		
		
	\end{frontmatter}
	
	\section{Introduction} \label{sec:introduction}
	\noindent
	In solid state physics, one of the most popular models to describe electron dynamics is the Schr\"{o}dinger
	equation in the semiclassical regime
	\begin{equation}
	\left\{
	\begin{aligned}
	i\epsilon\partial_t\psi^\epsilon&=-\frac{\epsilon^2}{2}\Delta\psi^\epsilon+v^{\epsilon}(\bx) \psi^\epsilon,
	\quad \bx\in \mathbb{R}^d, \quad t\in \mathbb{R},\\
	\psi^\epsilon|_{t=0}&=\psi_{\textrm{in}}(\bx),\quad \bx\in \mathbb{R}^d,
	\end{aligned}
	\right.
	\label{eqn:Sch}
	\end{equation}
	where $ 0<\epsilon \ll1 $ is an effective Planck constant describing the microscopic and macroscopic scale ratio,
	$ d $ is the spatial dimension, $ v^{\epsilon}(\bx) $ is the given electrostatic potential,
	$\psi^\epsilon = \psi^\epsilon(t,\bx) $ is the wavefunction, and $ \psi_{\textrm{in}}(\bx) $ is the initial data. In the community of mathematics, there has been a long history of interest from both mathematical and
	numerical perspectives; see for example the review paper \cite{JinActa:2011} and references therein.
	
	In the simplest situation, $v^{\epsilon}(\bx) = u(\bx)$, where $u(\bx)$ is an (external) macroscopic potential.
	$ \psi^\epsilon(t,\bx) $ propagates oscillations with a wavelength of $\mathcal{O}(\epsilon)$,
	a uniform $L^2-$approximation of the wavefunction requires the spatial mesh size $h=o(\epsilon)$ and
	the time step $k=o(\epsilon)$ in finite element method (FEM) and finite difference method (FDM) \cite{BaoJinMarkowich:2002, JinActa:2011}.
	If the spectral time-splitting method is employed, a uniform $L^2-$approximation of the wavefunction requires
	the spatial mesh size $h=\mathcal{O}(\epsilon)$ and the time stepsize $k=o(\epsilon)$ \cite{BaoJinMarkowich:2002}.
	For a perfect crystal, in the presence of an external field,  $v^{\epsilon}(\bx) = v(\frac{\bx}{\epsilon}) + u(\bx)$,
	where $v(\frac{\bx}{\epsilon})$ describes the electrostatic interaction of ionic cores.
	A number of methods have been proposed by taking advantage of the periodic structure of $v(\frac{\bx}{\epsilon})$,
	such as the Bloch decomposition based time-splitting spectral method \cite{Huangetal:2007, Huang:2008},
	the Gaussian beam method \cite{Jin:08, Jin:2010, Qian:2010, Yin:2011}, and the frozen Gaussian approximation
	method \cite{DelgadilloLuYang:2016}. The Bloch decomposition based time-splitting spectral method requires
	a mesh strategy $h=\mathcal{O}(\epsilon)$ and $k=\mathcal{O}(1)$ for the uniform $L^2-$approximation of the
	wavefunction. The Gaussian beam method and the frozen Gaussian approximation method are based on asymptotic
	analysis, and thus are especially efficient when $\epsilon$ is very small.
	
	With recent developments in nanotechnology, a variety of material devices with tailored functionalities have been fabricated,
	such as heterojunctions, including the ferromagnet/metal/ferromagnet structure for giant megnetoresistance \cite{Zutic:2004},
	the silicon-based heterojunction for solar cells \cite{Louwenetal:2016}, and quantum metamaterials \cite{Quach:2011}.
	A basic feature of these devices is the combination of dissimilar crystalline structures, which results
	a heterogeneous interaction from ionic cores with different lattice structures. Therefore, when travelling through
	a device, electrons experience a potential $v^{\epsilon}(\bx)$ which cannot be written in the abovementioned form.
	Consequently, all the available methods based on asymptotic analysis cannot be applied. Moreover, direct methods,
	such as FEM and FDM, are extremely inefficient with strong mesh size restrictions. This motivates us to design efficient numerical methods for \eqref{eqn:Sch} in the general situation.
	
	Our work is motivated by the multiscale FEM for solving elliptic problems with multiscale coefficients \cite{HouWuCai:99,EfendievHou:09}. The multiscale FEM is capable of correctly capturing the large scale components of the multiscale solution on a coarse grid without accurately resolving all the small scale features in the solution. This is accomplished by incorporating the local microstructures of the differential operator into the finite element basis functions. We remark that in the past four decades, many other efficient methods have been developed for the multiscale PDEs in the literature; see \cite{Babuska:94,Hughes:1998,ChenHou:02,Jenny:03,EngquistE:03,Kevrekidis:2003,OwhadiZhang:07,HanZhangCMS:12} for example and references therein.
	
	Recently, several works relevant to the compression of elliptic operator with heterogeneous and highly varying coefficients have been proposed. In \cite{Peterseim:2014}, Malqvist and Peterseim construct localized multiscale basis functions using a modified variational multiscale method. The exponentially decaying property of these modified basis has been shown both theoretically and numerically. Meanwhile, Owhadi \cite{Owhadi:2015,Owhadi:2017} reformulates the multiscale problem from the perspective of decision theory using the idea of gamblets as the modified basis. In particular, a coarse space of measurement functions is constructed from Bayesian perspective, and the gamblet space is explicitly constructed. In addition, the gamblets are still proven to decay exponentially such that localized computation is made possible. Hou and Zhang \cite{hou2017sparse} extend these works such that localized basis functions can also be constructed for higher-order strongly elliptic operators.
	
	In this paper, we propose a multiscale FEM to solve the Schr\"{o}dinger equation in the semiclassical regime. The localized multiscale basis are constructed using sparse compression of the Hamiltonian operator, and thus are ``blind" to the specific form of the potential. After an one-shot eigendecomposition, we can solve the resulting system of ordinary differential equations explicitly for the time evolution. In our approach, $ H=\mathcal{O}(\epsilon) $ and $ k$ is independent of $\epsilon$. Numerical examples in one dimension with a periodic potential, a multiplicative two-scale potential, and a layered potential, and in two dimension with an additive two-scale potential and a checkboard potential are tested to demonstrate the robustness and efficiency of the proposed method. Moreover, first-order and second-order rates of convergence are observed in $H^1$ and $L^2$ norms, respectively.
	
	The rest of the paper is organized as follows. In \secref{sec:MsFEM}, we introduce a multiscale FEM for the semiclassical Schr\"{o}dinger equation and discuss the properties of the proposed method. Numerous numerical results are presented in \secref{sec:NumericalExamples}, including both one dimensional and two dimensional examples to demonstrate the efficiency of the proposed method. Conclusions and discussions are drawn in \secref{sec:Conclusion}.
	
	\section{ Multiscale finite element method for the semiclassical Schr\"{o}dinger equation} \label{sec:MsFEM}
	\noindent
	In this section, we construct the multiscale finite element basis functions based on an optimization approach, and use these basis functions as the approximation space in the Galerkin method to solve the Schr\"{o}dinger equation. A couple of properties of the proposed method are also given.
	
	\subsection{Construction of multiscale basis functions} \label{sec:OC}
	\noindent
	Recall that the Schr\"{o}dinger equation \eqref{eqn:Sch} is defined in $\mathbb{R}^d$. However, numerically we can only deal with
	bounded domains, thus artificial boundary condition is needed here. For the sake of brevity, we shall restrict ourselves to a bounded
	domain with prescribed boundary condition. In fact, artificial boundary condition can also be combined with the proposed approach which will
	be investigated in a subsequent work. Therefore we consider the following problem
	\begin{equation}
	\left\{
	\begin{aligned}
	i\epsilon\partial_t\psi^\epsilon&=-\frac{\epsilon^2}{2}\Delta\psi^\epsilon+v^{\epsilon}(\bx) \psi^\epsilon,\quad \bx\in D,\quad t\in \mathbb{R},\\
	\psi^\epsilon &\in H_{\textrm{P}}^{1}(D),\\
	\psi^\epsilon|_{t=0}&=\psi_{\textrm{in}}(\bx).
	\end{aligned}
	\right.
	\label{Sch}
	\end{equation}
	Here $ D = [0,1]^d $ is the spatial domain and $ H_{\textrm{P}}^{1}(D) = \{\psi |\psi\in H^1(D) \textrm{ and } \psi \textrm{ is periodic over D} \}$. $ \psi_{\textrm{in}}(\bx) $ is the initial data over $D$. Define the Hamiltonian operator $\mathcal{H}(\cdot)\equiv -\frac{\epsilon^2}{2}\Delta(\cdot)+v^{\epsilon}(\bx)(\cdot) $ and introduce the following energy notation $ ||\cdot||_V $ for Hamiltonian operator
	\begin{align}\label{eqn:energynorm}
		||\psi^\epsilon||_V=\frac12(\mathcal{H}\psi^\epsilon,\psi^\epsilon)=\frac12\int_{D}\frac{\epsilon^2}{2}|\nabla \psi^\epsilon|^2+v^{\epsilon}(\bx)|\psi^\epsilon|^2 \d\bx.
	\end{align}
	
	Note that \eqref{eqn:energynorm} does not define a norm since $v^{\epsilon}$ usually can be negative, and thus the bilinear form
	associated to this notation is not coercive, which is quite different from the case of elliptic equations. However, this does not
	mean that available approaches \cite{HouWu:97, BabuskaLipton:2011, Peterseim:2014, Owhadi:2017, hou2017sparse} cannot be used for
	the Schr\"{o}dinger equation. In fact, we shall utilize the similar idea to construct localized multiscale finite element basis
	functions on a coarse mesh by an optimization approach using the above energy notation $ ||\cdot||_V $ for the Hamiltonian operator.
	
	To construct such localized basis functions,  we first partition the physical domain $D$ into a set of regular coarse elements with mesh size $H$. For example, we divide $D$ into a set of
	non-overlapping triangles $\mathcal{T}_{H}=\cup\{K\}$, such that no vertex of one triangle lies in the interior of the edge of another triangle. In each element $K$, we define a set of nodal basis $\{\varphi_{j,K},j=1,...,k\}$ with $k$ being the number of nodes of the element. From now on, we neglect the subscript $K$ for notational convenience. The functions $\varphi_{i}(\bx)$ are called measurement functions, which are chosen as the characteristic functions on each coarse element in \cite{hou2017sparse,Owhadi:2017} and piecewise linear basis functions in \cite{Peterseim:2014}. In \cite{LiZhangCiCP:18,hou2018model}, it
	is found that the usage of nodal basis functions reduces the approximation error and thus the same setting is adopted in the current work.
	
	Let $\mathcal{N}$ denote the set of vertices of $\mathcal{T}_{H}$ (removing the repeated vertices due to the periodic boundary condition) and $N_{x}$ be the number of vertices. For every vertex $\bx_i\in\mathcal{N}$, let $\varphi_{i}(\bx)$ denote the corresponding nodal basis function, i.e., $\varphi_{i}(\bx_j)=\delta_{ij}$. Since all the nodal basis functions $\varphi_{i}(\bx)$ are continuous across the boundaries of the elements, we have
	\begin{align*}
		V^{H}=\{\varphi_{i}(\bx):i=1,...,N_x \}\subset H_{\textrm{P}}^{1}(D),
	\end{align*}
	Then, we can solve optimization problems to obtain the multiscale basis functions.
	Specifically, let $\phi_i(\bx)$ be the minimizer of the following constrained optimization problem
	\begin{align}
		\phi_i& =\underset{\phi \in H_{\textrm{P}}^{1}(D)}{\arg\min} ||\phi||_V   \label{OC_SchGLBBasis_Obj}\\
		\text{s.t.}\ &\int_{D}\phi \varphi_{j} \d\bx= \delta_{i,j},\ \forall 1\leq j \leq N_{x}.  \label{OC_SchGLBBasis_Cons1}
	\end{align}
	The superscript $\epsilon$ is dropped for notational simplicity and the periodic boundary condition is incorporated into the above optimization problem through the solution space $H_{\textrm{P}}^{1}(D)$. With these multiscale finite element basis functions $ \{\phi_i(\bx)\}_{i=1}^{N_x} $, we can solve the Schr\"{o}dinger equation \eqref{Sch} using the Galerkin method.
	\begin{remark}
		Note that the energy notation $||\cdot||_V$ in \eqref{eqn:energynorm} does not define a norm. However, as long as $v^{\veps}(\bx)$ is
		bounded from below and the fine mesh size $h$ is small enough, the discrete problem of \eqref{OC_SchGLBBasis_Obj} - \eqref{OC_SchGLBBasis_Cons1}
		is convex and thus admits a unique solution; see \cite{hou2017sparse, LiZhangCiCP:18} for details.
	\end{remark}
	
	\subsection{Exponential decay of the multiscale finite element basis functions}
	\noindent
	We shall show that the multiscale basis functions $\{\phi_i(\bx)\}_{i=1}^{N_x}$ decay exponentially fast away from its associated vertex $x_i\in\mathcal{N}_{c}$ under certain conditions. This allows us to localize the basis functions to a relatively smaller domain and reduce the computational cost.
	
	In order to obtain localized basis functions, we first define a series of nodal patches $\{D_{\ell}\}$ associated with $\bx_i\in\mathcal{N}$ as
	\begin{align}
		D_{0}&:=\textrm{supp} \{ \varphi_{i} \} = \cup\{K\in\mathcal{T}_H | \bx_i \in K \}, \label{nodal_patch0}\\
		D_{\ell}&:=\cup\{K\in\mathcal{T}_H | K\cap \overline{D_{\ell-1}} \neq \emptyset\}, \quad  \ell=1,2,\cdots.
		\label{nodal_patchl}
	\end{align}	
	\begin{assumption} \label{CoarseMeshResolution}
		We assume that the potential term $v^{\epsilon}(\bx)$ is bounded, i.e.,  $V_0 := ||v^{\epsilon}(\bx)||_{L^\infty (D)}< +\infty$ and the   mesh size $H$ of $\mathcal{T}_{H}$ satisfies
		\begin{equation}
		\label{eqn:meshcondition}
		\sqrt{V_0}H/\epsilon\lesssim 1,
		\end{equation}
		where $\lesssim$ means bounded from above by a constant.
	\end{assumption}
	\noindent Under this assumption, many typical potentials in the Schr\"{o}dinger equation \eqref{Sch} can be treated as a perturbation to the kinetic operator. Thus, they can be computed using our method. Then, we can show that the multiscale finite element basis functions have the exponentially decaying property.
	\begin{proposition}[Exponentially decaying property]\label{ExponentialDecay}
		Under the resolution condition of the coarse mesh, i.e., \eqref{eqn:meshcondition}, there exist constants $C>0$ and $0<\beta<1$ independent of $H$, such that
		\begin{equation}\label{eqn:exponentialdecay}
		||\nabla \phi_i(\bx) ||_{D\backslash D_{\ell}}\leq C \beta^{\ell} ||\nabla \phi_i(\bx) ||_{D},
		\end{equation}
		for any $i=1, 2, ..., N_x$.
	\end{proposition}

Proof of \eqref{eqn:exponentialdecay} will be given in \cite{ChenMaZhang:prep}. The main idea is to combine an iterative Caccioppoli-type argument \cite{Peterseim:2014, LiZhangCiCP:18} and some refined estimates with respect to $\veps$. To demonstrate the exponentially decaying property of multiscale basis functions, we use the multiscale basis function centered
at $x=1/2$ in \textbf{Example \ref{example2}} for a sequence of $\veps$ from $1/40$ to $1/160$ for illustration. 
The left figure plots $\nabla\phi / ||\phi||_{L_2}$ with respect to the distance to $x=1/2$, which shows both the exponential decay and the $\veps$
dependence with respect to the distance. The right figure plots $ E_{\textrm{relative}}=\dfrac{||\nabla \phi ||_{D}-||\nabla \phi ||_{D_{\ell}}}{\max(||\nabla \phi ||_{D}-||\nabla \phi ||_{D_{\ell}})}$ with respect to the patch size $\ell$, which shows the decay rate of $E_{\textrm{relative}}$ with respect
to $\ell$ is independent of $\veps$, and thus the estimate in \eqref{eqn:exponentialdecay} is sharp. 
$||\phi||_{L_2}$ and $\max(||\nabla \phi ||_{D}-||\nabla \phi ||_{D_{\ell}})$ in the denominators are used such that
$\nabla\phi / ||\phi||_{L_2}$ and $ E_{\textrm{relative}}$ with respect to $\veps$ are in a similar range of magnitudes.
	\begin{figure}[H]
		\centering
		\begin{subfigure}{0.39\textwidth}
			\centering
			\includegraphics[width=\textwidth]{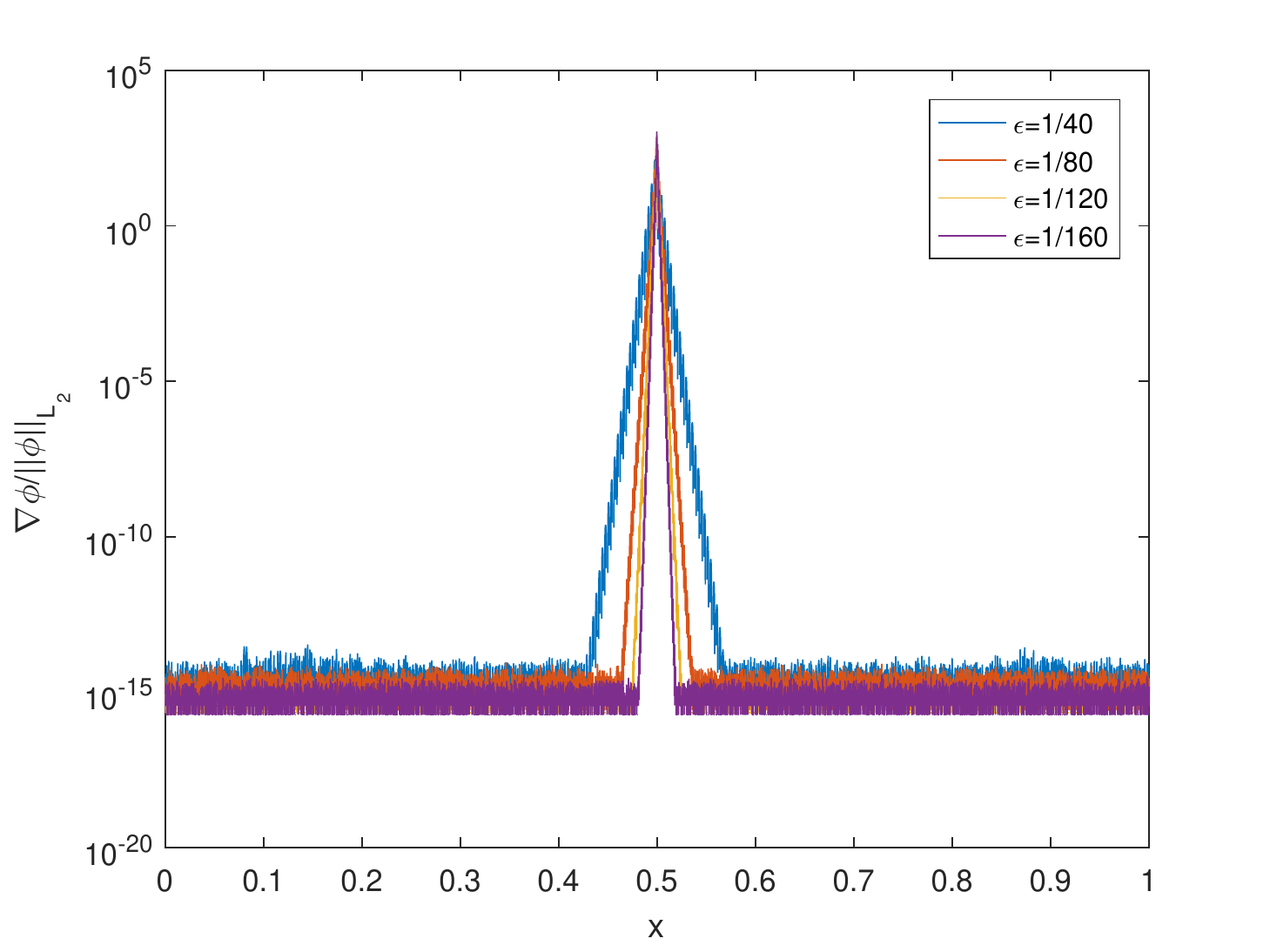}
			\label{fig:gradient}
		\end{subfigure}
		\begin{subfigure}{0.39\textwidth}
			\centering
			\includegraphics[width=\textwidth]{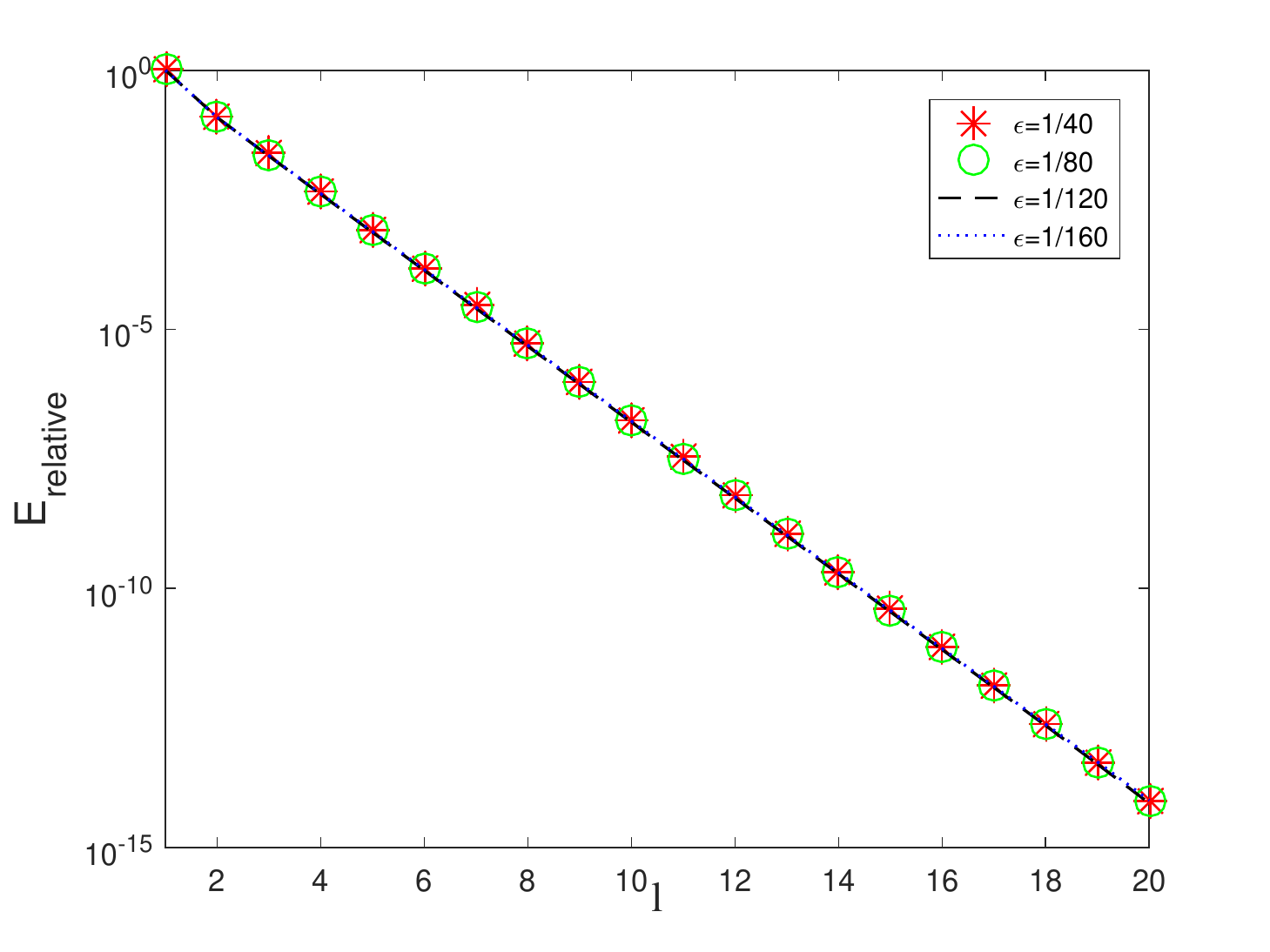}
			\label{fig:exponential}
		\end{subfigure}
		\caption{Exponentially decaying properties of the multiscale basis function in \textbf{Example \ref{example2}} for a sequence of $\veps$
from $1/40$ to $1/160$. Left: $\nabla\phi / ||\phi||_{L_2}$ with respect to the distance to $x=1/2$; Right: $ E_{\textrm{relative}}=\dfrac{||\nabla \phi ||_{D}-||\nabla \phi ||_{D_{\ell}}}{\max(||\nabla \phi ||_{D}-||\nabla \phi ||_{D_{\ell}})}$ with respect to the patch size $\ell$.
It is observed that the exponentially decaying property of $E_{\textrm{relative}}$ with respect to $\ell$ is independent of $\veps$,
which implies the sharpness of \eqref{eqn:exponentialdecay}.}
		\label{exponential_decay}
	\end{figure}
		
The exponential decay of the basis functions enables us to localize the support sets of
	the basis functions $\{\phi_i(\bx)\}_{i=1}^{N_x}$, so that the corresponding stiffness matrix is sparse and the
	computational cost is reduced. In practice, we define a modified constrained optimization problem as follows
	\begin{align}
		\phi_i^{\textrm{loc}}& =\underset{\phi \in H_{\textrm{P}}^{1}(D)}{\arg\min} ||\phi||_V   \label{OC_SchBasis_Obj}\\
		\text{s.t.}\ &\int_{D_{l^*}}\phi \varphi_{j} \d\bx= \delta_{i,j},\ \forall 1\leq j \leq N_{x}, \label{OC_SchBasis_Cons1}\\
		&\phi(\bx)= 0, \ x\in D\backslash D_{l^*},\label{OC_SchBasis_Cons2}
	\end{align}
	where $D_{l^*}$ is the support set of the localized multiscale basis function $\phi_i^{\textrm{loc}}(\bx)$ and the choice of the integer $l^*$ depends on the decaying speed of $\phi_i^{\textrm{loc}}(\bx)$. In \eqref{OC_SchBasis_Cons1} and \eqref{OC_SchBasis_Cons2}, we have used the fact that $\phi_i(\bx)$ has the exponentially decaying property so that we can localize the support set of $\phi_i(\bx)$ to a smaller domain $D_{l^*}$. In numerical experiments, we find that a small integer $l^*\sim \log(L/H)$  will give accurate results, where $L$ is the diameter of domain $D$. Moreover, the optimization problem \eqref{OC_SchBasis_Obj} - \eqref{OC_SchBasis_Cons2} can be solved in parallel. Therefore, the exponentially decaying property significantly reduces our computational cost in constructing basis functions and computing the solution of the Schr\"{o}dinger equation \eqref{Sch}.
	
	\subsection{Time marching} \label{sec:NumCom}
	\noindent	
	With the localized multiscale finite element basis functions $ \{\phi_i^{\textrm{loc}}(\bx)\}_{i=1}^{N_x}$, we can approximate the wavefunction
	by $ \psi^\epsilon(\bx,t)=\sum_{i=1}^{N_x}c_i(t)\phi_i^{\textrm{loc}}(\bx) $ using the Galerkin method. Therefore, the coefficients $ c_i(t),i=1,...,N_x $
	satisfies a system of ordinary differential equations (ODEs).
	
	In details, we use localized multiscale finite element basis functions for both the test space and the trial space in the weak formulation for \eqref{Sch}
	\begin{align}
		\left(i\epsilon\partial_t\sum_{i=1}^{N_x}c_i(t)\phi_i^{\textrm{loc}}(\bx),\phi_j^{\textrm{loc}}\right)=\left(\mathcal{H}\phi_i^{\textrm{loc}},\phi_j^{\textrm{loc}}\right),\quad \bx\in D,t\in \mathbb{R},j=1,\cdots,N_x.
		\label{appro_Sch}
	\end{align}
	Let
	\begin{align*}
		S_{i,j}&=\int_{D}\nabla \phi_i^{\textrm{loc}} \cdot \nabla \phi_j^{\textrm{loc}} \d\bx, \\
		M_{i,j}&=\int_{D} \phi_i^{\textrm{loc}} \phi_j^{\textrm{loc}} \d\bx, \\
		V_{i,j}&=\int_{D}\phi_i^{\textrm{loc}} v^{\epsilon}(\bx) \phi_j^{\textrm{loc}} \d\bx.
	\end{align*}
	Then, we can formulate \eqref{appro_Sch} as follows
	\begin{align}\label{eqn:Schmatrix}
		i\epsilon M\bc_t=(\frac{\epsilon^2}{2}S+V)\bc,
	\end{align}
	in which $ \bc=(c_1(t),c_2(t),...,c_{N_x}(t))^T $ and $ \bc_t=(\partial_tc_1(t),\partial_tc_2(t),...,\partial_tc_{N_x}(t))^T $.
	For illustration, we further rewrite \eqref{eqn:Schmatrix} as
	\begin{align}\label{eqn:Schcompact}
		\bc_t=\frac{1}{i\epsilon} B\bc
	\end{align}
	with $ B=M^{-1}A $ and $A=\frac{\epsilon^2}{2}S+V$. Note that both $M$ and $A$ are symmetric, so the eigenvalues of $B$ are real.
	However, since $B$ is not symmetric in general, it admits an eigendecomposition or a Jordan canonical form.
	
	\subsubsection{Eigendecomposition}
	\noindent
	In this case, $ B $ can be factorized as $ B=P\Lambda P^{-1} $, where $\Lambda $ is the diagonal matrix and $P$ is an invertible matrix.
	Substituting this form into \eqref{eqn:Schcompact} yields
	\begin{align*}
		\bc_t&=\frac{1}{i\epsilon}P\Lambda P^{-1}\bc,
	\end{align*}
	which can be further rewritten as
	\begin{align}\label{ODEs1}
		\tilde{\bc}_t&=\frac{1}{i\epsilon}\Lambda \tilde{\bc}
	\end{align}
	with $ \tilde{\bc}=P^{-1}\bc $. Since $\Lambda$ is diagonal, \eqref{ODEs1} can be solved explicitly.
	
	For example, consider the temporal interval of interest to be $[0, 1]$ and denote $ \tilde{\bc}_{n}=\tilde{\bc}(t_n) $, $ \bc_{n}=\bc(t_n)=(c_1(t_n),c_2(t_n),...,c_{N_x}(t_n))^T $ and $t_n = nk =1$, in which $ k$ denotes the time stepsize. We have
	\begin{align*}
		\tilde{\bc}_{n}&=\exp{\left(\int_{0}^{1}\frac{1}{i\epsilon}\Lambda dt\right)}\tilde{\bc}_0=\exp{\left(\frac{1}{i\epsilon}\Lambda\right)}\tilde{\bc}_0,
	\end{align*}
	and
	\begin{align}
		\bc_n&=P\exp{\left(\frac{1}{i\epsilon}\Lambda\right)}P^{-1}\bc^{\textrm{in}}.
		\label{eig_analytic}
	\end{align}
	Here $ \tilde \bc_0=P^{-1}\bc^{\textrm{in}} $ and $ \bc^{\textrm{in}}=(c^{\textrm{in}}_1,c^{\textrm{in}}_2,...,c^{\textrm{in}}_{N_x})^T $.
	$ \bc^{\textrm{in}}$ is obtained from the initial data $ \psi_{\textrm{in}}(\bx)=\sum_{i=1}^{N_x}c^{\textrm{in}}_i\phi_i^{\textrm{loc}}(\bx)$ using Galerkin projection, i.e.,
	\begin{align}\label{eqn:icGalerkin}
		\left(\psi_{\textrm{in}},\phi_j^{\textrm{loc}}\right) = \left(\sum_{i=1}^{N_x} c^{\textrm{in}}_i\phi_i^{\textrm{loc}},\phi_j^{\textrm{loc}}\right),\quad j=1,\cdots, N_x.
	\end{align}
	\eqref{eqn:icGalerkin} can be written in a compact form as $M\bc^{\textrm{in}} = \boldsymbol{\psi}^\textrm{in}$ with $\boldsymbol{\psi}^\textrm{in} = \left((\psi_{\textrm{in}},\phi_1^{\textrm{loc}}),\cdots, (\psi_{\textrm{in}},\phi_{N_x}^{\textrm{loc}})\right)^T$. Solving this linear system of equations
	produces $\bc^{\textrm{in}}$.
	
	\subsubsection{Jordan canonical form}
	\noindent	
	If $B$ is not diagonalizable, then there exists an invertible matrix $P$, such that $B$ can be factorized as
	$B=PJP^{-1} $, where $ J $ has the block diagonal form
	$$J=
	\left(
	\begin{matrix}
	J_1      & 0      & \cdots & 0      \\
	0      & J_2     & \cdots & 0      \\
	\vdots & \vdots & \ddots & \vdots \\
	0      & 0      & \cdots & J_s      \\
	\end{matrix}
	\right)
	$$
	with $J_i$ the $i-$th Jordan block associated to the corresponding eigenvalue $ \lambda_i$
	$$J_i=
	\left(
	\begin{matrix}
	\lambda_i      & 1      & \cdots & 0      \\
	0      & \lambda_i     & \ddots & \vdots     \\
	\vdots & \vdots & \ddots & 1 \\
	0      & 0      & \cdots & \lambda_i     \\
	\end{matrix}
	\right).
	$$
	For $i=1,\cdots, s$, $ r_i $, the multiplicity of $\lambda_i$ satisfies $ r_1+r_2+...+r_s=N_x $.
	
	We now proceed with explicit time marching. Similar to \eqref{ODEs1}, we have
	\begin{align}
		\tilde{\bc}_t&=\frac{1}{i\epsilon}J \tilde{\bc} \label{ODEs2}
	\end{align}
	Since $J$ is almost diagonal, we can solve the system of ODEs \eqref{ODEs2} in a similar fashion
	but with a bit more complexity.
	
	Take the $j-$th block $ J_j $ for example. According to \eqref{ODEs2}, we have for the $j-$th block
	\begin{equation}
	\left(
	\begin{array}{cccc}
	\partial_t\tilde c_{j_1}(t)\\
	\partial_t\tilde c_{j_2}(t)\\
	... \\
	\partial_t\tilde c_{j_{r_j}(t)}
	\end{array}
	\right)
	=\frac{1}{i\epsilon}
	\left(
	\begin{matrix}
	\lambda_j      & 1      & \cdots & 0      \\
	0      & \lambda_j     & \ddots & \vdots     \\
	\vdots & \vdots & \ddots & 1 \\
	0      & 0      & \cdots & \lambda_j     \\
	\end{matrix}
	\right)
	\left(
	\begin{array}{cccc}
	\tilde c_{j_1}(t)\\
	\tilde c_{j_2}(t)\\
	... \\
	\tilde c_{j_{r_j}(t)}
	\end{array}
	\right).
	\label{blockode}
	\end{equation}
	We can solve \eqref{blockode} in a backward manner explicitly by first solving for $\tilde c_{j_{r_j}(t)}$, then $\tilde c_{j_{r_{j-1}}(t)}$, $\cdots$, until $\tilde c_{j_{r_1}(t)}$. Over the time interval $[0, t_n]$, this procedure results
	\begin{equation}
	\left(
	\begin{array}{cccc}
	\tilde c_{j_1}(t_n)\\
	\tilde c_{j_2}(t_n)\\
	... \\
	\tilde c_{j_{r_j}(t_n)}
	\end{array}
	\right)
	=e^{\frac{1}{i\epsilon}\lambda_j t_n}
	\left(
	\begin{matrix}
	1      & \frac{t_n}{i\epsilon}    & \cdots &(\frac{t_n}{i\epsilon})^{j-1}\frac{1}{(j-1)!}   \\
	0      & 1    & \ddots & \vdots     \\
	\vdots & \vdots & \ddots & \frac{t_n}{i\epsilon}  \\
	0      & 0      & \cdots & 1    \\
	\end{matrix}
	\right)
	\left(
	\begin{array}{cccc}
	\tilde c_{j_1}(0)\\
	\tilde c_{j_2}(0)\\
	... \\
	\tilde c_{j_{r_j}(0)}
	\end{array}
	\right)
	\label{blockode_result}
	\end{equation}
	Repeating this procedure for each Jordan block,	we can obtain  $ \tilde{\bc}_{n} $ and $ \bc_{n}=P\tilde{\bc}_{n} $.
	
	There are a couple of properties of the proposed method. The first is the gauge invariance.
	If a discrete gauge transformation $\boldsymbol{d}_n = \bc_n \exp{\left(\frac{i}{\veps}\omega t_n \right)}$
	is introduced to \eqref{eqn:Schcompact}, it is easy to check such a transformation commutes with $B$. Therefore,
	the current method is gauge-invariant.
	
	\begin{proposition}[Gauge invariance]\label{prop:gaugeinvariance}
		The multiscale finite element method is gauge-invariant.
	\end{proposition}
	
	\begin{proposition}[Conservation of total mass and total energy]\label{prop:conservation}
		The multiscale finite element method conserves both total mass and total energy, i.e.,
		\begin{align}
			||\psi^{n+1}||_{L^2} & = ||\psi^{n}||_{L^2}, \quad \forall n\ge 0, \label{eqn:mass}\\
			||\psi^{n+1}||_V & = ||\psi^{n}||_V, \quad \forall n\ge 0. \label{eqn:energy1}
		\end{align}
	\end{proposition}
	
	\begin{proof}
		By definition, $ \psi_{n}=\sum_{i=1}^{N_x}c_i(t_{n})\phi^{\textrm{loc}}_i$ and $ \psi_{n+1}=\sum_{i=1}^{N_x}c_i(t_{n+1})\phi^{\textrm{loc}}_i$.
		Then, we have
		\begin{align}
			||\psi^{n+1}||_{L^2}^2 & = \sum_{1\leq i,j\leq N_x} c^*_i(t_{n+1}) c_j(t_{n+1}) \left(\phi^{\textrm{loc}}_i,\phi^{\textrm{loc}}_j\right) = \bc_{n+1}^* M \bc_{n+1},\label{eqn:density1}\\
			||\psi^{n}||_{L^2}^2 & = \sum_{1\leq i,j\leq N_x} c^*_i(t_{n}) c_j(t_{n}) \left(\phi^{\textrm{loc}}_i,\phi^{\textrm{loc}}_j\right) = \bc_{n}^* M \bc_{n}.\label{eqn:density2}
		\end{align}
		To avoid the detailed discussion of using eigendecomposition or Jordan canonical form, we use \eqref{eqn:Schcompact} and have
		\begin{equation}
		\bc_{n+1} = \exp{\left(\frac{k}{i\veps}M^{-1}A\right)}\bc_n.\label{eqn:density3}
		\end{equation}
		Substituting \eqref{eqn:density3} into \eqref{eqn:density1} yields
		\begin{align*}
			\bc_{n+1}^* M \bc_{n+1} & = \bc_{n}^* \exp{\left(-\frac{k}{i\veps}AM^{-1}\right)} M \exp{\left(\frac{k}{i\veps}M^{-1}A\right)} \bc_{n}
			= \bc_{n}^* M \bc_{n},
		\end{align*}
		which validates the conservation of total mass.
		Here we have used the facts that both $M$ and $A$ are real symmetric matrices and
		\[
		\exp{\left(-\frac{k}{i\veps}AM^{-1}\right)} M \exp{\left(\frac{k}{i\veps}M^{-1}A\right)} = M.
		\]
		
		Similarly, for total energy, we have
		\begin{align*}
			||\psi^{n+1}||_{V}^2 & = \frac12\sum_{1\leq i,j\leq N_x} c^*_i(t_{n+1}) c_j(t_{n+1}) \left(\mathcal{H}\phi^{\textrm{loc}}_i,\phi^{\textrm{loc}}_j\right) = \frac12\bc_{n+1}^* A \bc_{n+1}, \\
			||\psi^{n}||_{V}^2 & = \frac12\sum_{1\leq i,j\leq N_x} c^*_i(t_{n}) c_j(t_{n}) \left(\mathcal{H}\phi^{\textrm{loc}}_i,\phi^{\textrm{loc}}_j\right) = \frac12\bc_{n}^* A \bc_{n}.
		\end{align*}
		Using the fact that
		\[
		\exp{\left(-\frac{k}{i\veps}AM^{-1}\right)} A \exp{\left(\frac{k}{i\veps}M^{-1}A\right)} = A,
		\]
		we have
		\begin{align*}
			\bc_{n+1}^* A \bc_{n+1}  = \bc_{n}^*\exp{\left(-\frac{k}{i\veps}AM^{-1}\right)} A \exp{\left(\frac{k}{i\veps}M^{-1}A\right)}\bc_{n}
			= \bc_{n}^* A \bc_{n},
		\end{align*}
		which completes the proof the conservation of total energy.
	\end{proof}
	
	\section{Numerical examples}\label{sec:NumericalExamples}
	\noindent	
	In this section, we will test the proposed method for a number of examples in one dimension with a periodic potential, a multiplicative two-scale potential, and a layered potential, and in two dimension with an additive two-scale potential
	and a checkboard potential. Note that \textbf{Example \ref{example1}} and \textbf{Example \ref{example4}} can be solved
	by the approaches in \cite{Huangetal:2007, DelgadilloLuYang:2016}, while \textbf{Example \ref{example2}},
	\textbf{Example \ref{example3}}, and \textbf{Example \ref{example5}} cannot.
	In all cases, we denote $ \psi^{\epsilon}_{\textrm{exact}} $ the reference solution obtained by the
	Crank-Nicolson scheme in time with a very small stepsize and the standard FEM in space with a very small meshsize.
	We denote $ \psi^{\epsilon}_{\textrm{num}} $ the numerical solution obtained by our method. The computational domain
	$D=[0,1]$ in 1D and $D=[0,1]\times[0,1]$ in 2D and the final time $T=1$ in all examples. In all examples, both the
	total mass and the total energy are checked to be a constant during the time evolution.
	
	The initial data in 1D and 2D are chosen as
	\begin{equation}\label{eqn:ic1d}
	\psi_{\textrm{in}}(x)=(\frac{10}{\pi})^{1/4}e^{-5(x-1/2)^2},
	\end{equation}
	and
	\begin{equation}\label{eqn:ic2d}
	\psi_{\textrm{in}}(x,y)=(\frac{10}{\pi})^{1/2}e^{-5(x-1/2)^2-5(y-1/2)^2},
	\end{equation}
	respectively.
	
	For convenience, we introduce the $ L^2 $ norm and $ H^1 $ norm as
	\[
	||\psi^{\epsilon}||^2_{L^2}=\int_{D}|\psi^{\epsilon}|^2 \d\bx, \quad ||\psi^{\epsilon}||^2_{H^1}=\int_{D}|\nabla \psi^{\epsilon}|^2 \d\bx + \int_{D}|\psi^{\epsilon}|^2 \d\bx.
	\]
	In what follows, we shall compare the relative error between the numerical solution and the exact solution in both $ L^2 $ norm and $ H^1 $ norm
	\begin{align}
		\textrm{Error}_{L^2}& =\dfrac{||\psi^{\epsilon}_{\textrm{num}}-\psi^{\epsilon}_{\textrm{exact}}||_{L^2}}
		{||\psi^{\epsilon}_{\textrm{exact}}||_{L^2}}, \label{eqn:errorl2}\\
		\textrm{Error}_{H^1}& =\dfrac{||\psi^{\epsilon}_{\textrm{num}}-\psi^{\epsilon}_{\textrm{exact}}||_{H^1}}
		{||\psi^{\epsilon}_{\textrm{exact}}||_{H^1}}.\label{eqn:errorh1}
	\end{align}
	For \eqref{eqn:Sch} with \eqref{eqn:ic1d} or \eqref{eqn:ic2d}, we have $||\psi^{\epsilon}_{\textrm{exact}}||_{L^2}=1, \forall t>0$.
	Thus, the relative $L^2$ error \eqref{eqn:errorl2} is the same as the absolute $L^2$ error recorded in \cite{Huangetal:2007, DelgadilloLuYang:2016}. However, as $\veps$ reduces, $||\psi^{\epsilon}_{\textrm{exact}}||_{H^1}$ increases significantly.
	For example, when $\veps=1/256$, $||\psi^{\epsilon}_{\textrm{exact}}||_{H^1} = 179.93$ in \textbf{Example \ref{example3}}.
	We therefore consider relative errors in both $L^2$ norm and $H^1$ norm.
	
	Moreover, we will show the performance of our method for the approximation of observables, including the position density
	\begin{align}\label{eqn:density}
		n^{\epsilon}(\bx,t)=|\psi^{\epsilon}(\bx,t)|^2,
	\end{align}
	and the energy density
	\begin{align}\label{eqn:energy}
		e^{\epsilon}(\bx,t)=\frac{\epsilon^2}{2}|\nabla \psi^{\epsilon}(\bx,t)|^2+v^{\epsilon}(\bx) |\psi^{\epsilon}(\bx,t)|^2.
	\end{align}
	
	\begin{example}[1D case with a periodic potential]\label{example1}
		We start with the so-called Mathieu model where $ v^{\epsilon}(x) =\cos(2\pi\frac{x}{\epsilon}) $ is a periodic function of $x/\epsilon$.
		
		In Table \ref{table2} Table \ref{table_256case1}, we record the relative $L^2$ and $H^1$ errors on a series of coarse meshes when
		$ H/\epsilon = 1/2, 1/4, 1/8, 1/16, 1/32$ with
		$ \epsilon=1/40 $ and $ \epsilon=1/256 $, respectively. For a given $\veps$, one can easily see that the relative errors
		reduces in both $L^2$ norm and $H^1$ norm as $H/\veps$ reduces. Therefore, the meshsize condition \eqref{eqn:meshcondition} ($H/\veps \lesssim 1$) is necessary in our
		method to obtain numerical results with reasonable approximation accuracy.
		To get quantitative results, we further calculate the convergence rates in both norms. In Table \ref{table2}, convergence rates do not
		seem to be uniform in both norms when we change the value of $H/\veps$. However, as we further reduce $\veps$ to $1/256$, the results
		look better. Results in Table \ref{table_256case1} suggest that our method converges with rates 2 and 1 in $L^2$ norm and $H^1$ norm,
		respectively.
		\begin{table}[htbp]
			\centering
			\begin{tabular}{|r|r|r|r|r|}
				\hline
				$ H/\epsilon $ & $ \textrm{Error}_{L^2} $ & Order & $ \textrm{Error}_{H^1} $ & Order \\
				\hline
				$ 1/2 $ & 0.03763392 &       & 1.85952337 &  \\
				$ 1/4 $ & 0.03484993 & 0.12  & 1.73102295 & 0.10 \\
				$ 1/8 $ & 0.00037858 & 6.86  & 0.03425796 & 5.67 \\
				$ 1/16 $ & 0.00009600 & 1.99  & 0.02152483 & 0.67 \\
				$ 1/32 $ & 0.00004249 & 1.18  & 0.01661115 & 0.37 \\
				\hline
			\end{tabular}%
			\caption{Relative $L^2$ and $H^1$ errors of the wavefunction for \textbf{Example \ref{example1}} when $ \epsilon=1/40 $.}
			\label{table2}%
		\end{table}%

		\begin{table}[htbp]
			\centering
			\begin{tabular}{|r|r|r|r|r|}
				\hline
				$ H/\epsilon $ & $ \textrm{Error}_{L^2} $ & Order & $ \textrm{Error}_{H^1} $ & Order \\
				\hline
				$ 1/2 $ & 0.01562512 &       & 0.36076801 &  \\
				$ 1/4 $ & 0.00638978& 1.31  & 0.15041913 & 1.26 \\
				$ 1/8 $ & 0.00175078 & 1.86  & 0.04059514 & 1.89 \\
				$ 1/16 $ &0.00002774 & 6.45  & 0.00117276 & 5.17 \\
				$ 1/32 $ & 0.00000389 & 2.86  & 0.00058418 & 1.00 \\
				\hline
			\end{tabular}%
			\caption{Relative $L^2$ and $H^1$ errors of the wavefunction for \textbf{Example \ref{example1}} when $ \epsilon=1/256 $.}
			\label{table_256case1}%
		\end{table}%

		Next, we visualize profiles of the position density function \eqref{eqn:density} and the energy density function \eqref{eqn:energy}
		when $\veps=1/40$ and $\veps=1/256$ in
		Figure \ref{fig1} and Figure \ref{fig1_256}, respectively. Excellent agreements between numerical solutions and the exact solutions
		also imply that our method can also approximate the observables with high accuracy on coarse meshes with $ H=\mathcal{O}(\epsilon) $.
		\begin{figure}[tbph]
			\centering
			\begin{subfigure}{0.39\textwidth}
				\centering
				\includegraphics[width=\textwidth]{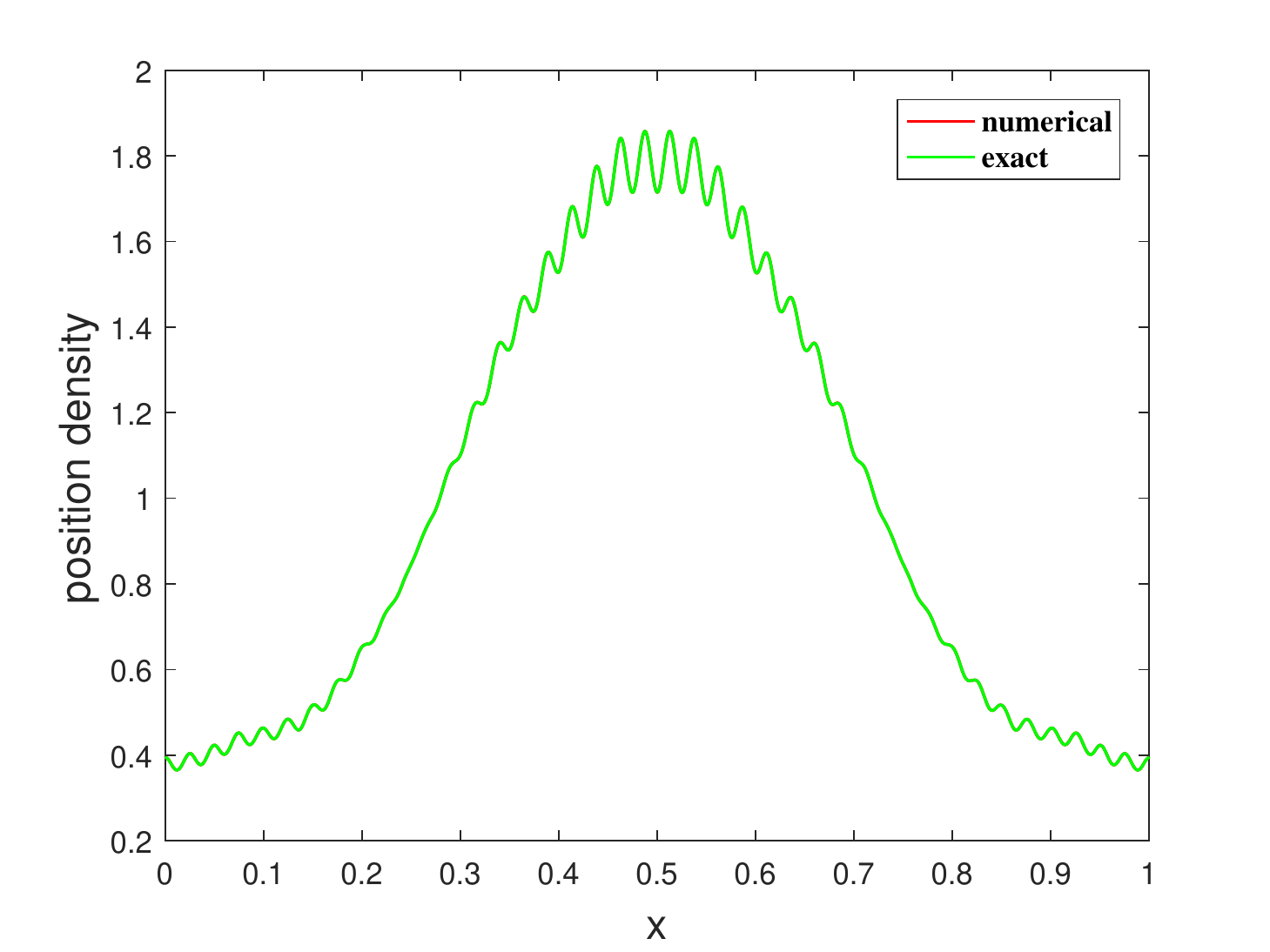}
				\label{fig:position1}
			\end{subfigure}
			\begin{subfigure}{0.39\textwidth}
				\centering
				\includegraphics[width=\textwidth]{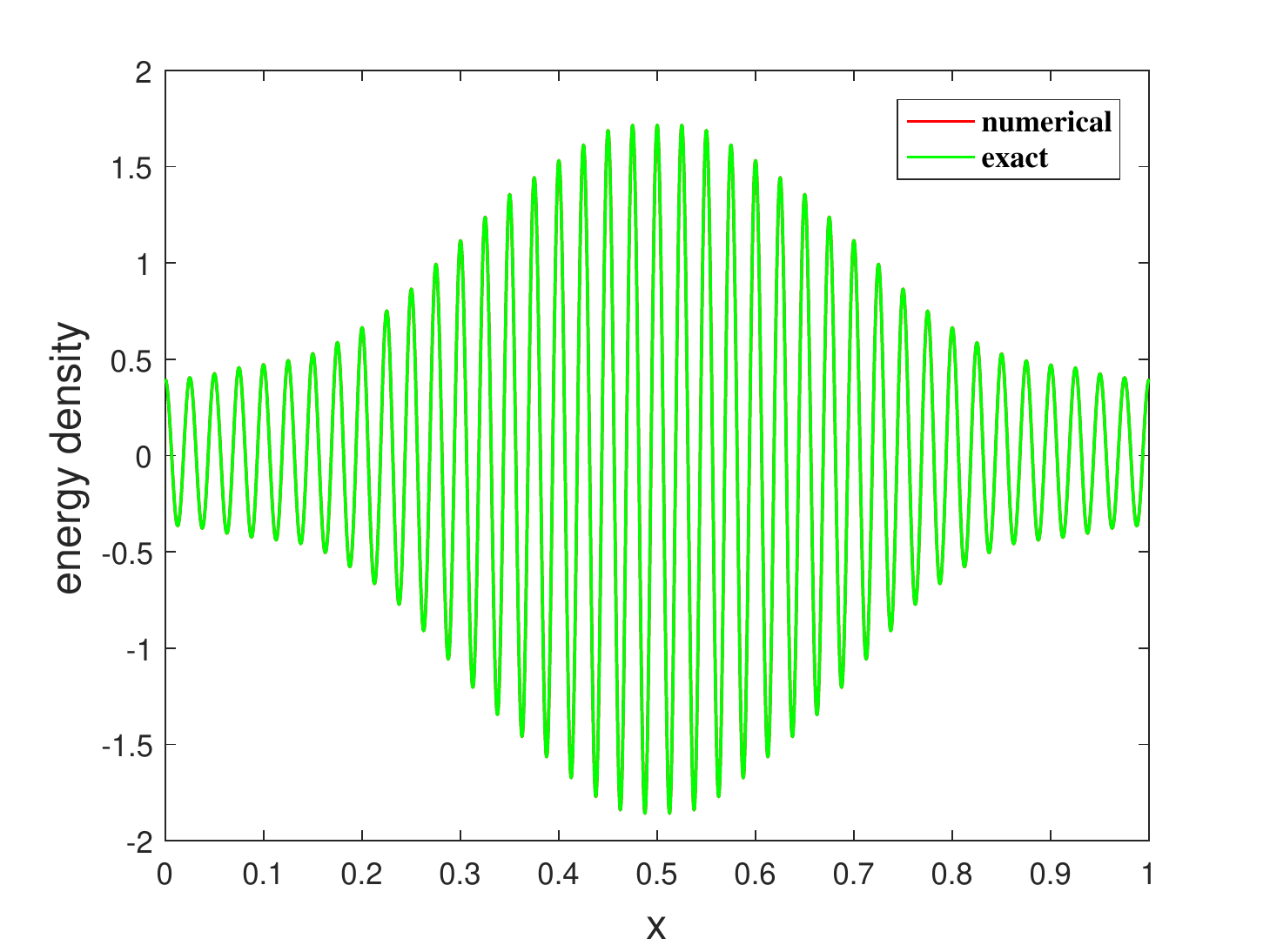}
				\label{fig:energy1}
			\end{subfigure}
			\caption{Profiles of numerical and exact density functions for \textbf{Example \ref{example1}} when $ H/\epsilon=1/16$ and $ \epsilon=1/40 $. Left: position density; Right: energy density.}
			\label{fig1}
		\end{figure}

		\begin{figure}[H]
			\centering
			\begin{subfigure}{0.39\textwidth}
				\centering
				\includegraphics[width=\textwidth]{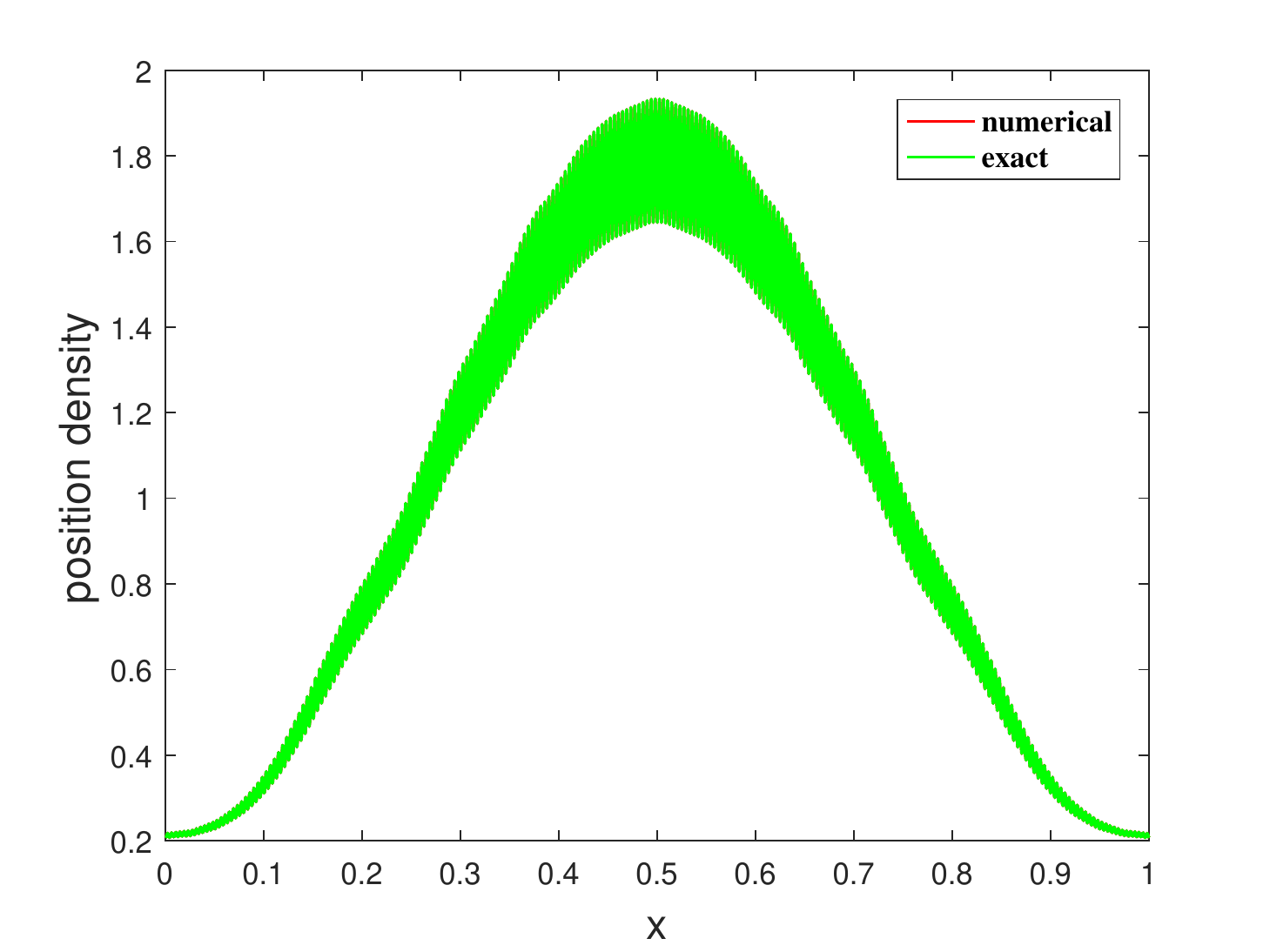}
				\label{fig:position1_256}
			\end{subfigure}
			\begin{subfigure}{0.39\textwidth}
				\centering
				\includegraphics[width=\textwidth]{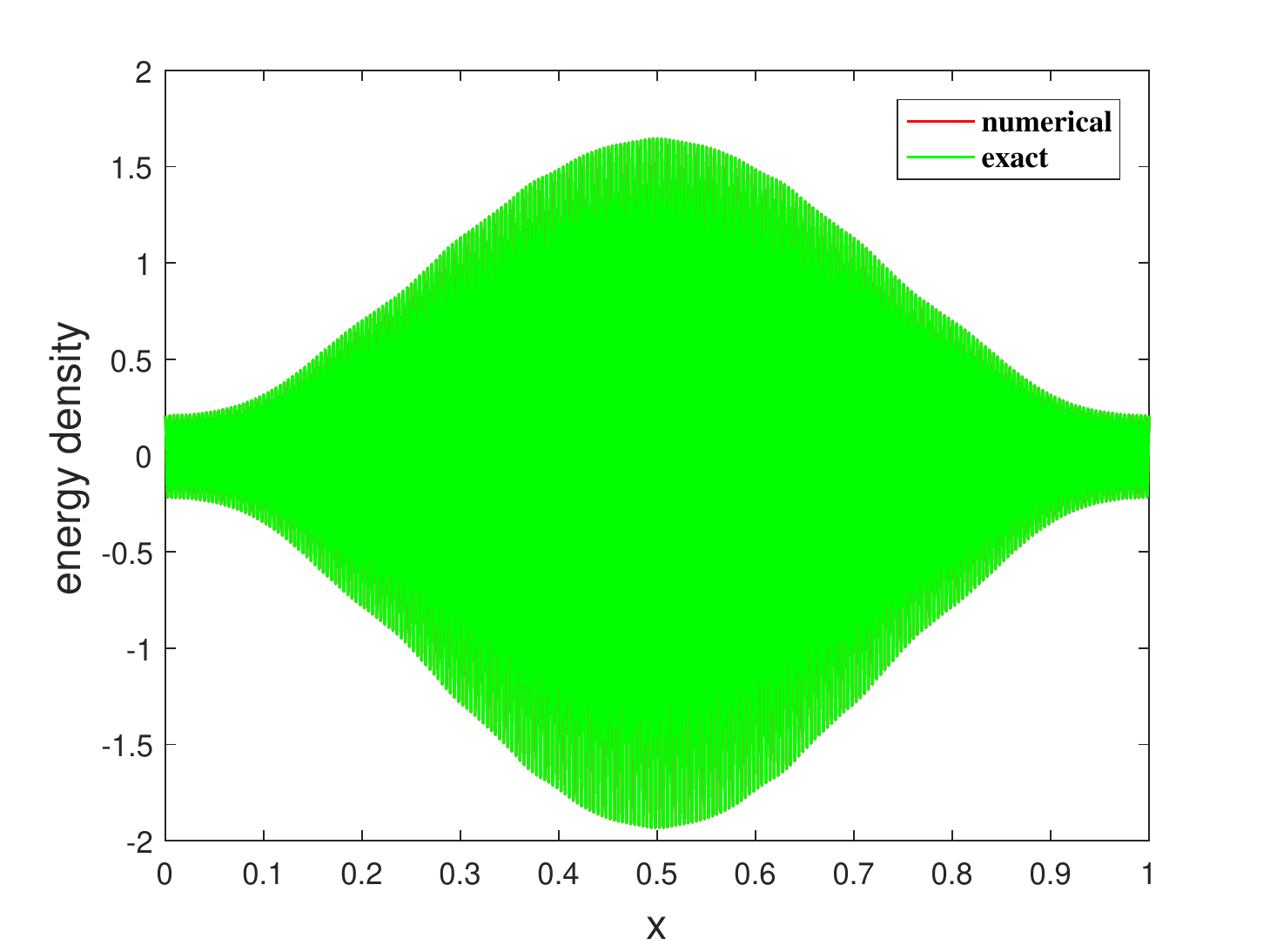}
				\label{fig:energy1_256}
			\end{subfigure}
			\caption{Profiles of numerical and exact density functions for \textbf{Example \ref{example1}} when $ H/\epsilon =1/16$ and $ \epsilon=1/256 $. Left: position density; Right: energy density.}
			\label{fig1_256}
		\end{figure}
		
	\end{example}

	\begin{example}[1D case with a multiplicative two-scale potential]\label{example2}
		The second example is a multiplicative two-scale potential where $v^{\epsilon}(x)=\sin(4x^2)\sin(2\pi \frac{x}{\epsilon})$.
		
		In Table \ref{table4} and Table \ref{table4_256}, we record the relative $L^2$ and $H^1$ errors on a series of coarse meshes when
		$ H/\epsilon = 1/2, 1/4, 1/8, 1/16, 1/32$ with $ \epsilon=1/40 $ and $ \epsilon=1/256 $, respectively.
		Similar to those observed in Table \ref{table2} Table \ref{table_256case1}, we have
		the meshsize condition \eqref{eqn:meshcondition} and convergence rates 2 and 1 in $L^2$ norm and $H^1$ norm,
		respectively.
		\begin{table}[htbp]
			\centering
			\begin{tabular}{|r|r|r|r|r|}
				\hline
				$ H/\epsilon $ & $ \textrm{Error}_{L^2} $ & Order & $ \textrm{Error}_{H^1} $ & Order \\
				\hline
				$ 	1/2 $ & 0.02416939 &       & 0.66267557 &  \\
				$ 1/4  $& 0.02505272 & -0.05  & 0.69203510 & -0.06 \\
				$ 1/8 $& 0.00036056& 6.38  & 0.02010799 & 5.11 \\
				$ 1/16 $ &0.00010290 & 1.84  & 0.01248123 & 0.70 \\
				$ 1/32 $ &0.00004129 & 1.31  & 0.00943716 & 0.41 \\
				\hline
			\end{tabular}%
			\caption{Relative $L^2$ and $H^1$ errors of the wavefunction for \textbf{Example \ref{example2}} when $ \epsilon=1/40 $.}
			\label{table4}%
		\end{table}%
		
		\begin{table}[H]
			\centering
			\begin{tabular}{|r|r|r|r|r|}
				\hline
				$ H/\epsilon $ & $ \textrm{Error}_{L^2} $ & Order & $ \textrm{Error}_{H^1} $ & Order \\
				\hline
				$ 	1/2 $ & 0.02524212 &       &0.62422624&  \\
				$ 1/4  $& 0.00623362 & 2.02  & 0.15528365& 2.01 \\
				$ 1/8 $& 0.00124272& 2.36  & 0.03098464& 2.37 \\
				$ 1/16 $ &0.00002386 & 5.95  & 0.00138226 & 4.49 \\
				$ 1/32 $ &0.00000478 & 2.34  & 0.00071506 & 0.96 \\
				\hline
			\end{tabular}%
			\caption{Relative $L^2$ and $H^1$ errors of the wavefunction for \textbf{Example \ref{example2}} when $ \epsilon=1/256 $.}
			\label{table4_256}%
		\end{table}%
		
		Profiles of the position density function \eqref{eqn:density} and the energy density function \eqref{eqn:energy}
		when $\veps=1/40$ and $\veps=1/256$ are plotted in
		Figure \ref{fig2} and Figure \ref{fig2_256}, respectively. Excellent agreements between numerical solutions and the exact solutions
		are observed again.
		\begin{figure}[tbph]
			\centering
			\begin{subfigure}{0.39\textwidth}
				\centering
				\includegraphics[width=\textwidth]{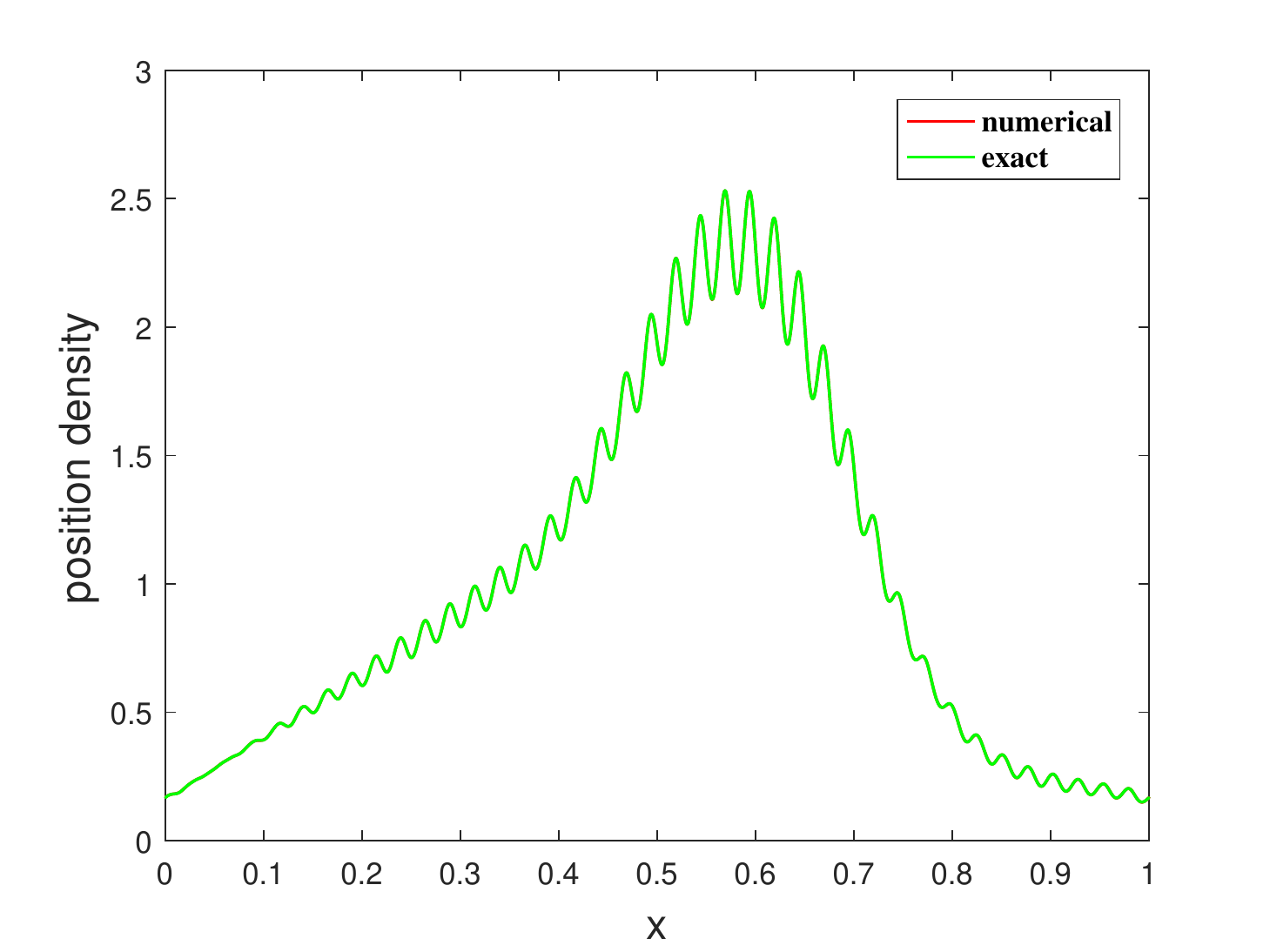}
				\label{fig:position2}
			\end{subfigure}
			\begin{subfigure}{0.39\textwidth}
				\centering
				\includegraphics[width=\textwidth]{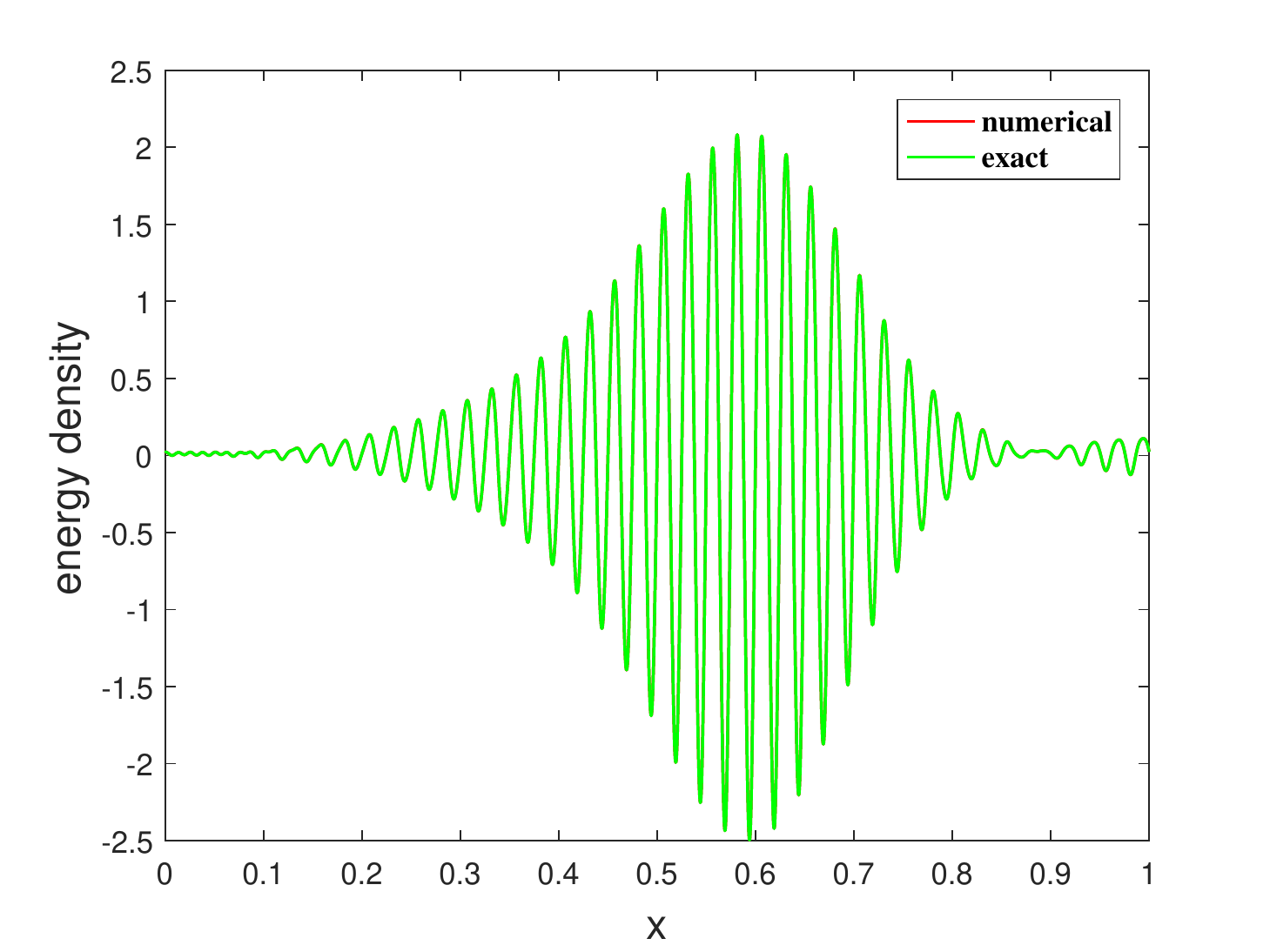}
				\label{fig:energy2}
			\end{subfigure}
			\caption{Profiles of numerical and exact density functions for \textbf{Example \ref{example2}} when $ H / \epsilon =1/8$ and $ \epsilon=1/40 $. Left: position density; Right: energy density.}
			\label{fig2}
		\end{figure}

		\begin{figure}[H]
			\centering
			\begin{subfigure}{0.39\textwidth}
				\centering
				\includegraphics[width=\textwidth]{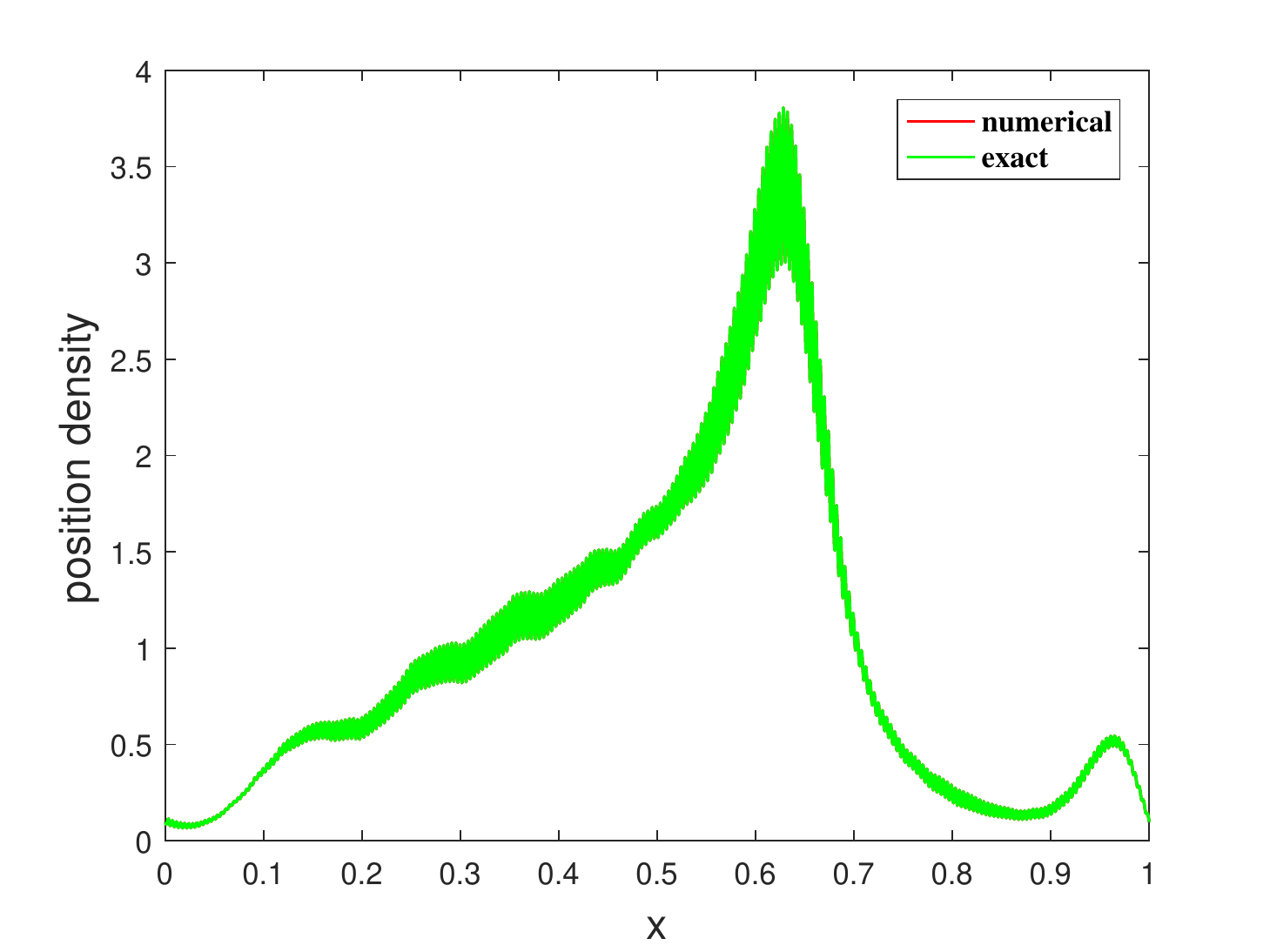}
				\label{fig:position2_256}
			\end{subfigure}
			\begin{subfigure}{0.39\textwidth}
				\centering
				\includegraphics[width=\textwidth]{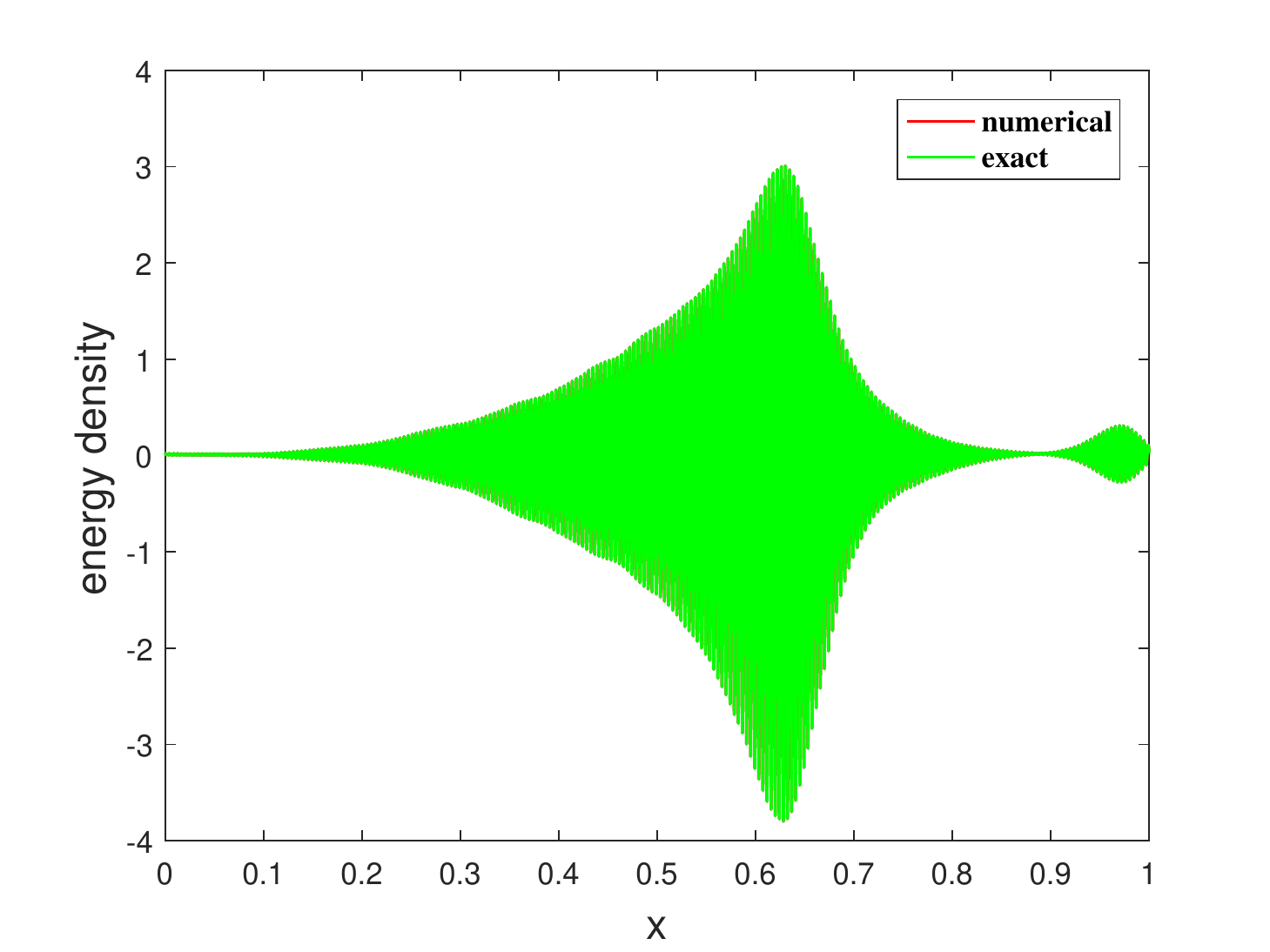}
				\label{fig:energy2_256}
			\end{subfigure}
			\caption{Profiles of numerical and exact density functions for \textbf{Example \ref{example2}} when $ H/\epsilon=1/16$ and $ \epsilon=1/256 $. Left: position density; Right: energy density.}
			\label{fig2_256}
		\end{figure}
		
	\end{example}

	\begin{example}[1D case with a layered potential]\label{example3}
		
		Consider
		\begin{equation}\label{eqn:potential3}
		v^{\epsilon}(x)= u(x) + \left\{
		\begin{aligned}
		\cos(2\pi\frac{x}{\epsilon_1})+1,\qquad&0\leq x\leq \frac{1}{3},\\
		\cos(2\pi\frac{x}{\epsilon_2})+1,\qquad&\frac{1}{3}< x\leq \frac{2}{3},\\
		\cos(2\pi\frac{x}{\epsilon_1})+1,\qquad&\frac{2}{3}< x\leq 1,
		\end{aligned}
		\right.
		\end{equation}
		where $ \epsilon_1=1/64 $ and $ \epsilon_2=1/256 $. This potential is used to mimic the heterojunction
		commonly used in spintronic devices \cite{Zutic:2004} in the presence of an external potential $ u(x)=|x-0.5|^2 $.
		Note that the potential \eqref{eqn:potential3} is set to be discontinuous to mimic the interface between
		dissimilar lattice structures.
		We set $\epsilon=1/256$.
		
		In Table \ref{table7}, we record the relative $L^2$ and $H^1$ errors on a series of coarse meshes when
		$ H/\epsilon = 1/2, 1/4, 1/8, 1/16, 1/32$ with $\epsilon=1/256$.
		Similar to those observed in Table \ref{table_256case1} and Table \ref{table4_256}, we have
		the meshsize condition \eqref{eqn:meshcondition} and convergence rates 2 and 1 in $L^2$ norm and $H^1$ norm,
		respectively.
		\begin{table}[H]
			\centering
			\begin{tabular}{|r|r|r|r|r|}
				\hline
				$ H/\epsilon $ & $ \textrm{Error}_{L^2} $ & Order & $ \textrm{Error}_{H^1} $ & Order \\
				\hline
				$ 	1/2 $ &0.04635974&       &0.35181488&  \\
				$ 1/4  $& 0.02329746& 0.99  &0.20526570& 0.78\\
				$ 1/8 $&0.00157223& 3.95  &0.01436669& 3.84\\
				$ 1/16 $ &0.00003767& 5.71  & 0.00077857& 4.22 \\
				$ 1/32 $ &0.00000527& 2.85  &0.00025634& 1.60\\
				\hline
			\end{tabular}%
			\caption{Relative $L^2$ and $H^1$ errors of the wavefunction for \textbf{Example \ref{example3}} when $ \epsilon=1/256 $.}
			\label{table7}%
		\end{table}
		
		Profiles of the position density function \eqref{eqn:density} and the energy density function \eqref{eqn:energy}
		when $\veps=1/256$ are plotted in
		Figure \ref{fig4}. Excellent agreements between numerical solutions and the exact solutions are observed again.
		\begin{figure}[H]
			\centering
			\begin{subfigure}{0.39\textwidth}
				\centering
				\includegraphics[width=\textwidth]{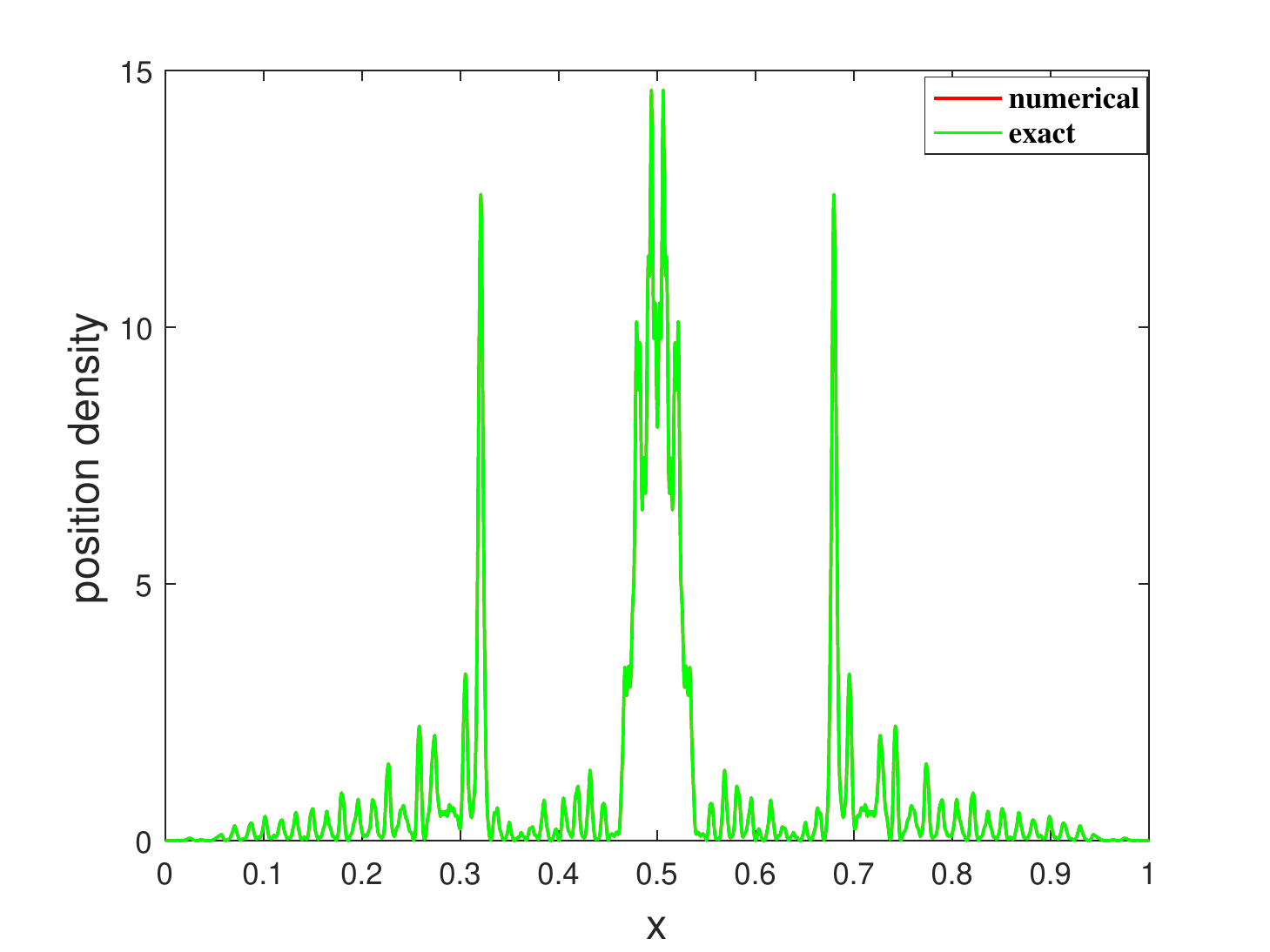}
				\label{fig:positionABA1D}
			\end{subfigure}
			\begin{subfigure}{0.39\textwidth}
				\centering
				\includegraphics[width=\textwidth]{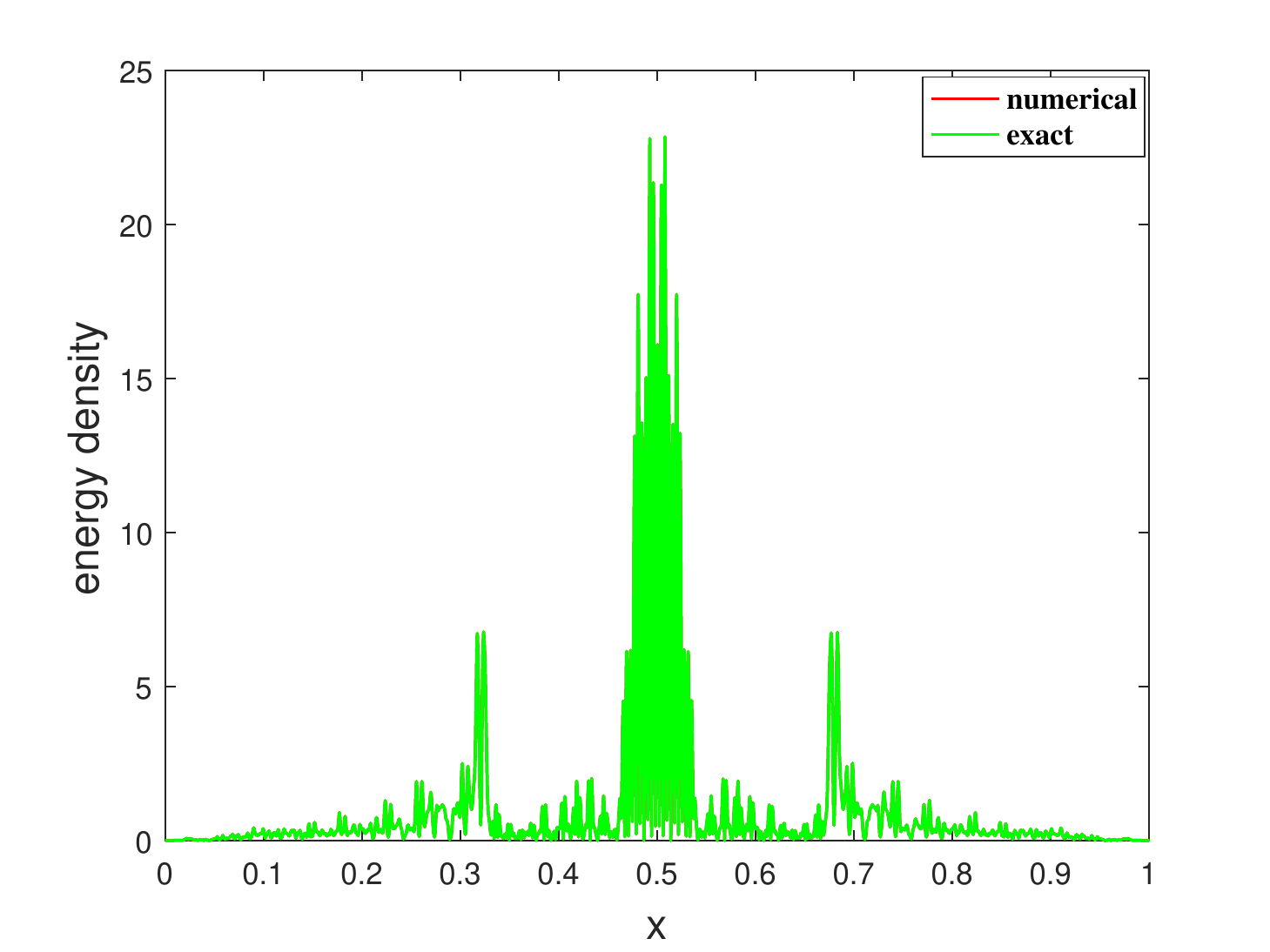}
				\label{fig:energyABA1D}
			\end{subfigure}
			\caption{Profiles of numerical and exact density functions for \textbf{Example \ref{example3}} when $ H / \epsilon = 1/16$ and $ \epsilon=1/256 $. Left: position density; Right: energy density.}
			\label{fig4}
		\end{figure}
		
	\end{example}

	\begin{example}[2D case with an additive two-scale potential]\label{example4}
		\noindent	
		The first 2D example is an additive two-scale potential of the form
		\[
		v^{\epsilon}(x,y) =  1+\sin(4x^2y^2)+\frac{(1.5+\sin(2\pi \frac{x}{\epsilon}))}{(1.5+\cos(2\pi \frac{y}{\epsilon}))}.
		\]
		
		Profiles of the position density function \eqref{eqn:density} and the energy density function \eqref{eqn:energy}
		when $\veps=1/16$ are plotted in Figure \ref{fig4} and Figure \ref{fig5}, respectively. Similar to those in 1D,
		excellent agreements between numerical solutions and the exact solutions are observed again.
		\begin{figure}[H]
			\centering
			\begin{subfigure}{0.39\textwidth}
				\centering
				\includegraphics[width=\textwidth]{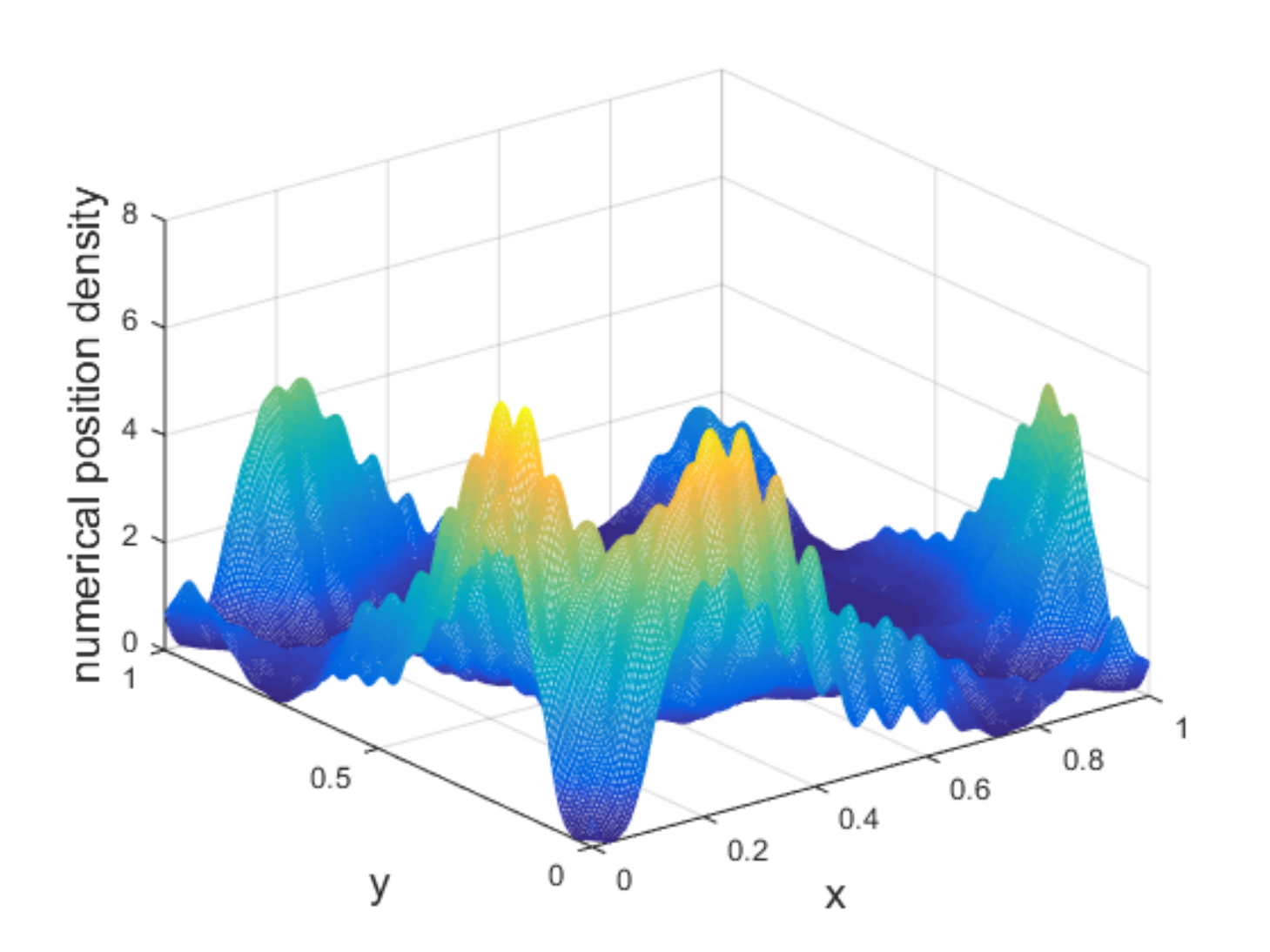}
				\label{fig:position4_1}
			\end{subfigure}
			\begin{subfigure}{0.39\textwidth}
				\centering
				\includegraphics[width=\textwidth]{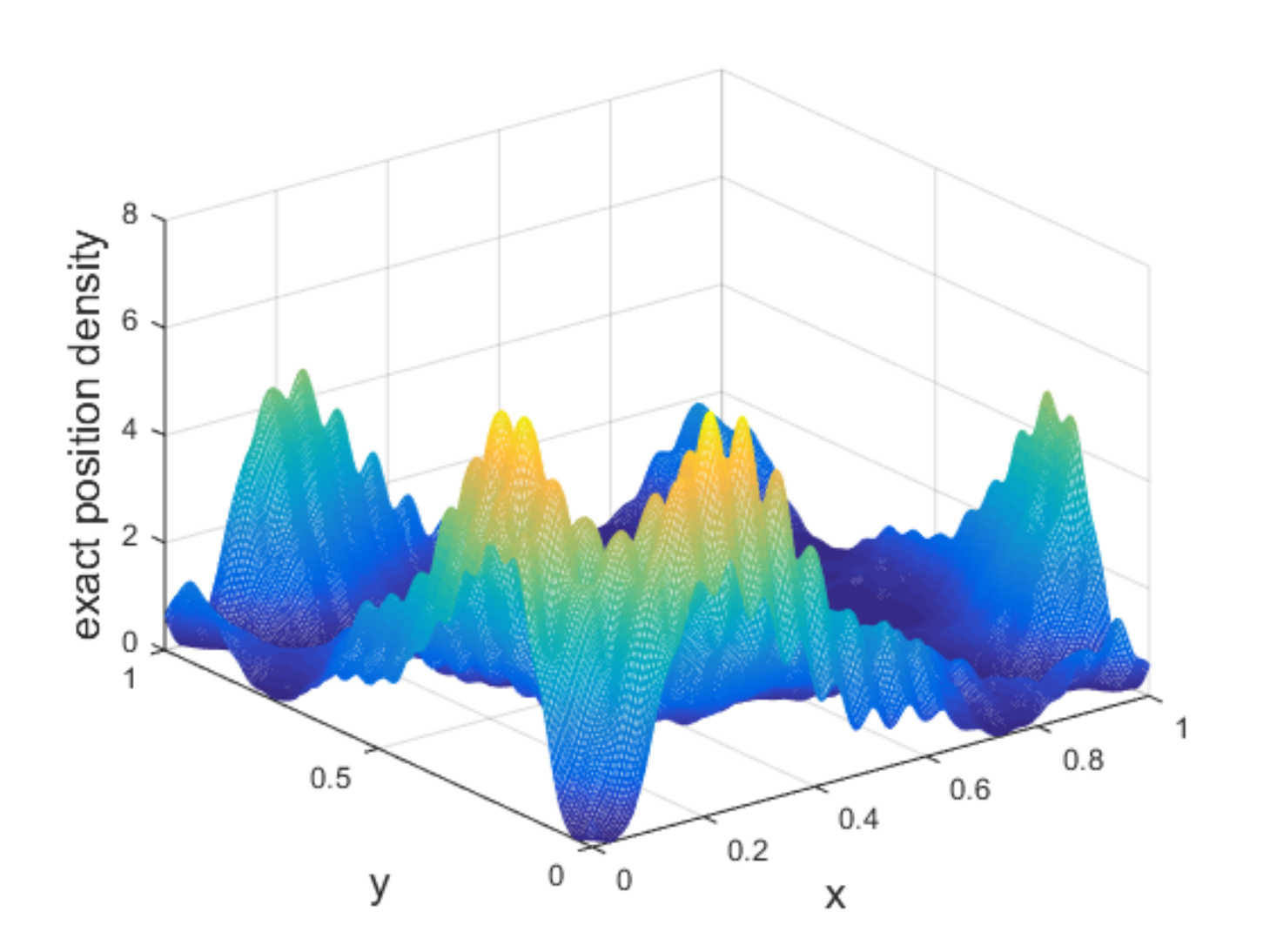}
				\label{fig:position4_2}
			\end{subfigure}
			\caption{Profiles of position density functions for \textbf{Example \ref{example4}} when $ H / \epsilon =1/4$ and $ \epsilon=1/16 $. Left: numerical solution; Right: exact solution.}
			\label{fig5}
		\end{figure}
		
		\begin{figure}[H]
			\centering
			\begin{subfigure}{0.39\textwidth}
				\centering
				\includegraphics[width=\textwidth]{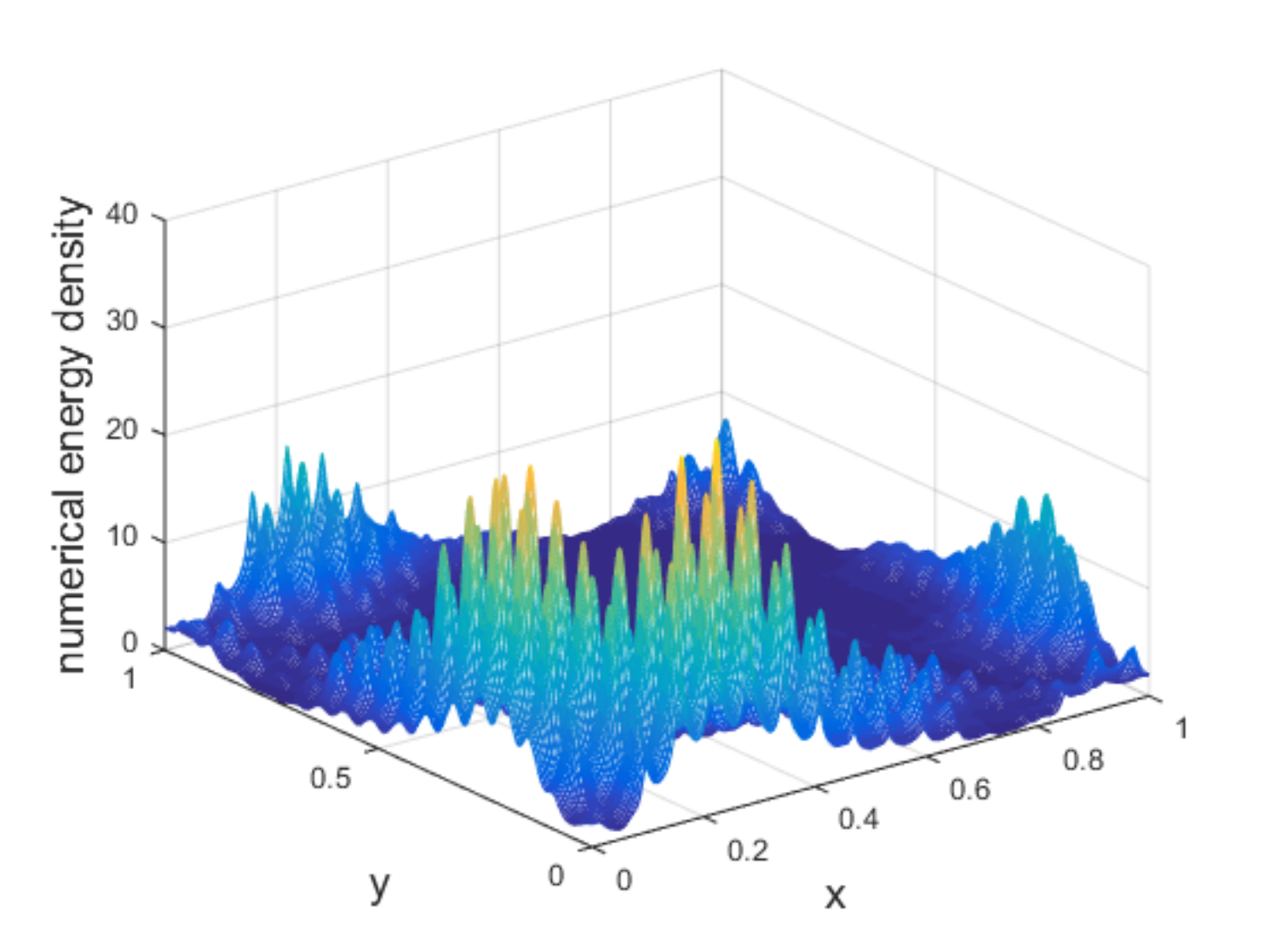}
				\label{fig:energy4_1}
			\end{subfigure}
			\begin{subfigure}{0.39\textwidth}
				\centering
				\includegraphics[width=\textwidth]{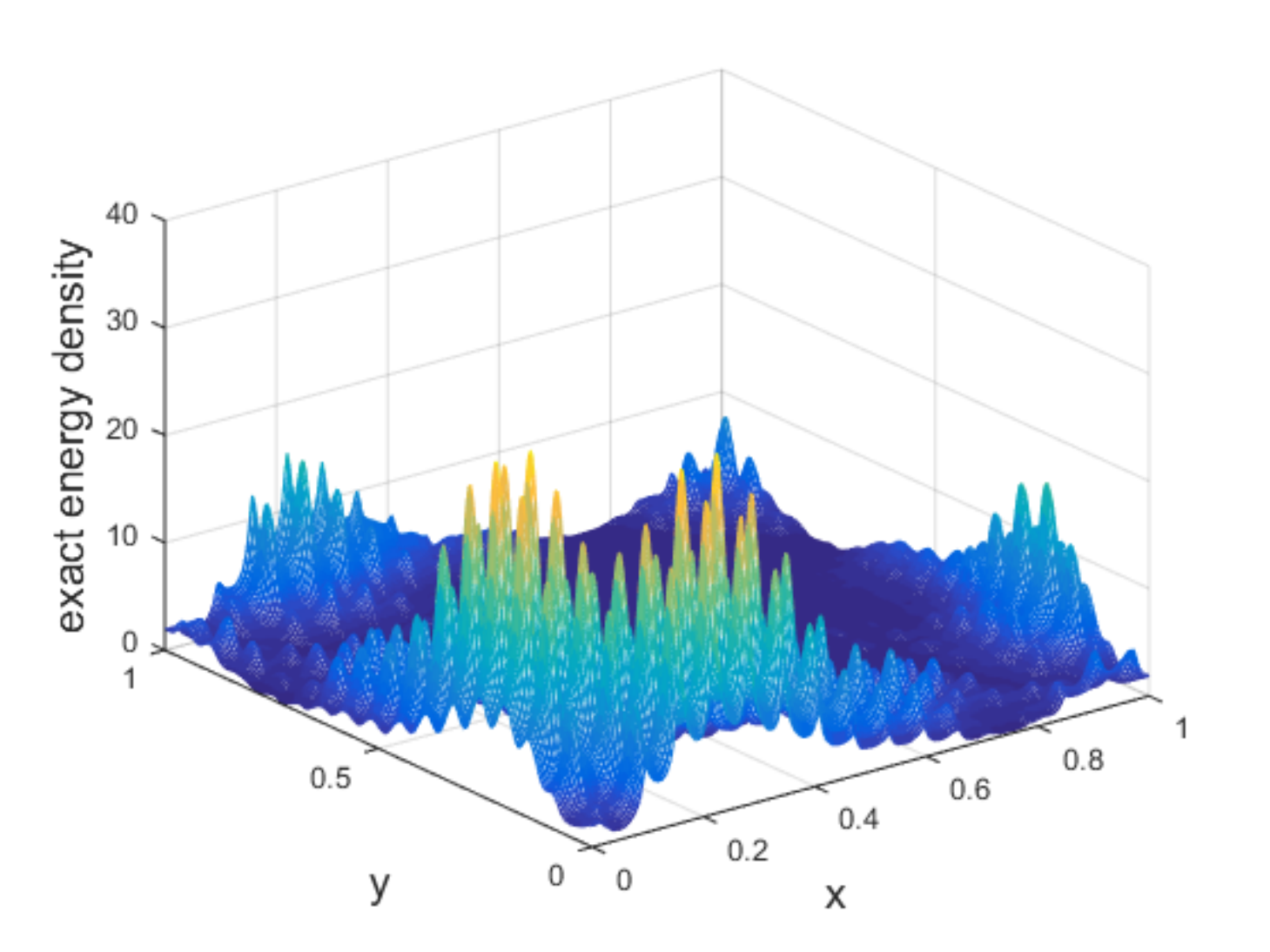}
				\label{fig:energy4_2}
			\end{subfigure}
			\caption{Profiles of energy density functions for \textbf{Example \ref{example4}} when $ H/\epsilon = 1/4$ and $ \epsilon=1/16 $. Left: numerical solution; Right: exact solution.}
			\label{fig6}
		\end{figure}
		
		Table \ref{table_example2Dcase1} records the relative errors in both $L^2$ norm and $H^1$ norm for a series of coarse meshes satisfying
		$ H/\epsilon = 1/2, 1/4, 1/8$ with $\epsilon=1/16$.
		\begin{table}[H]
			\centering
			\begin{tabular}{|r|r|r|r|r|}
				\hline
				$ H/\epsilon $ & $ \textrm{Error}_{L^2} $ & Order & $ \textrm{Error}_{H^1} $ & Order \\
				\hline
				$ 	1/2 $ &0.04462747&       &0.35583737&  \\
				$ 1/4  $&0.02760301& 0.69  &0.24731666& 0.52\\
				$ 1/8 $& 0.00459407& 2.59  &0.08146394& 1.60\\
				\hline
			\end{tabular}%
			\caption{Relative $L^2$ and $H^1$ errors of the wavefunction for \textbf{Example \ref{example4}} when $ \epsilon=1/16 $.}
			\label{table_example2Dcase1}
		\end{table}
		
		In Figure \ref{fig7} and Figure \ref{fig8}, we plot the real and imaginary parts of the numerical and exact wavefunctions, respectively.
		As confirmed in Table \ref{table_example2Dcase1}, convergence of the numerical wavefunction to the exact wavefunction
		is observed in both $L^2$ norm and $H^1$ norm. Again, the meshsize condition \eqref{eqn:meshcondition} and convergence rates
		2 and 1 in $L^2$ norm and $H^1$ norm are suggested.
		\begin{figure}[H]
			\centering
			\begin{subfigure}{0.39\textwidth}
				\centering
				\includegraphics[width=\textwidth]{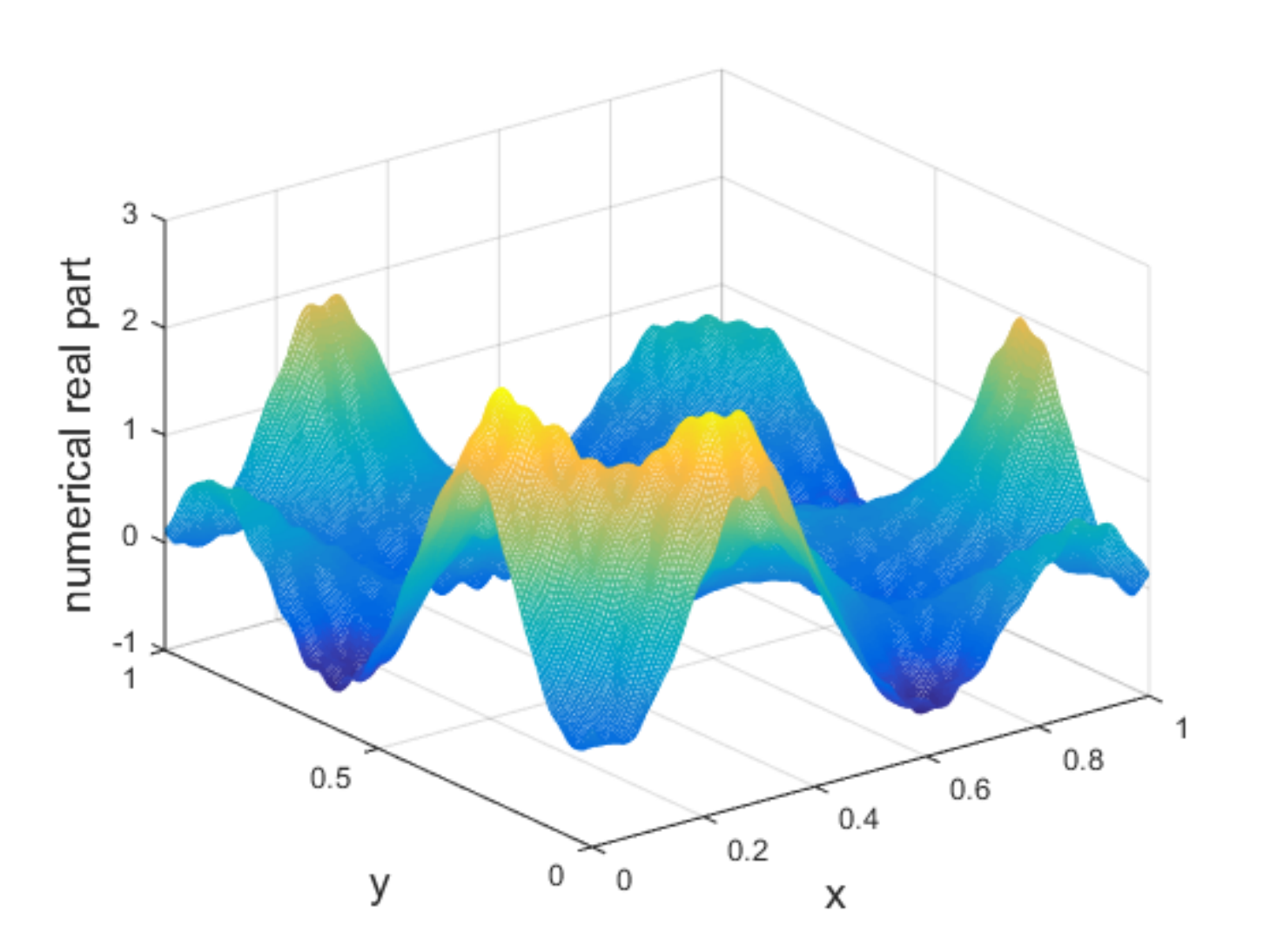}
				\label{fig:realpart_1}
			\end{subfigure}
			\begin{subfigure}{0.39\textwidth}
				\centering
				\includegraphics[width=\textwidth]{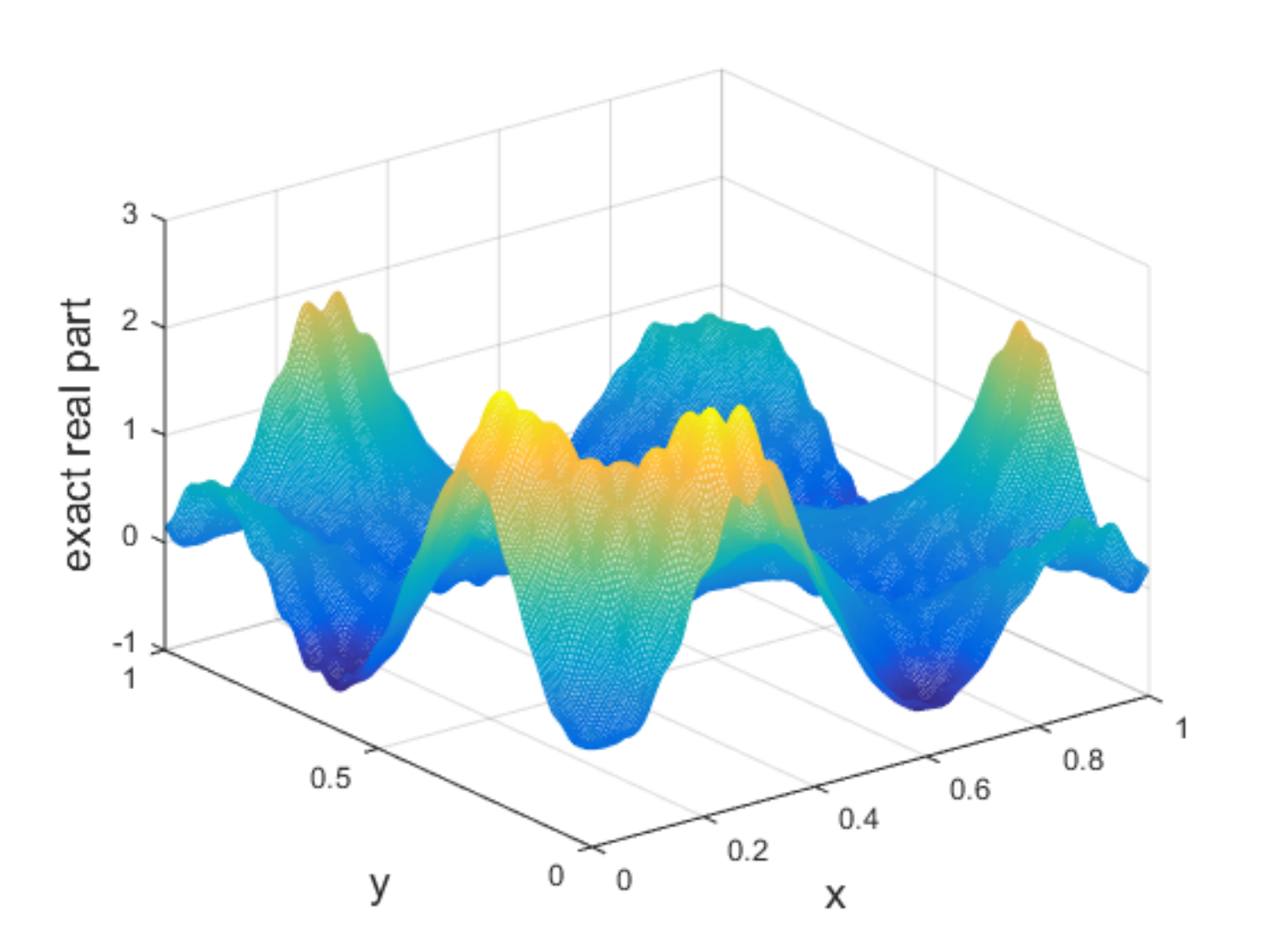}
				\label{fig:realpart_2}
			\end{subfigure}
			\caption{Real part of the wavefunction $ \psi^{\epsilon}(x,y,t) $ for \textbf{Example \ref{example4}} when $ H / \epsilon=1/4$ and $ \epsilon=1/16 $. Left: numerical solution; Right: exact solution.}
			\label{fig7}
		\end{figure}
		
		\begin{figure}[H]
			\centering
			\begin{subfigure}{0.39\textwidth}
				\centering
				\includegraphics[width=\textwidth]{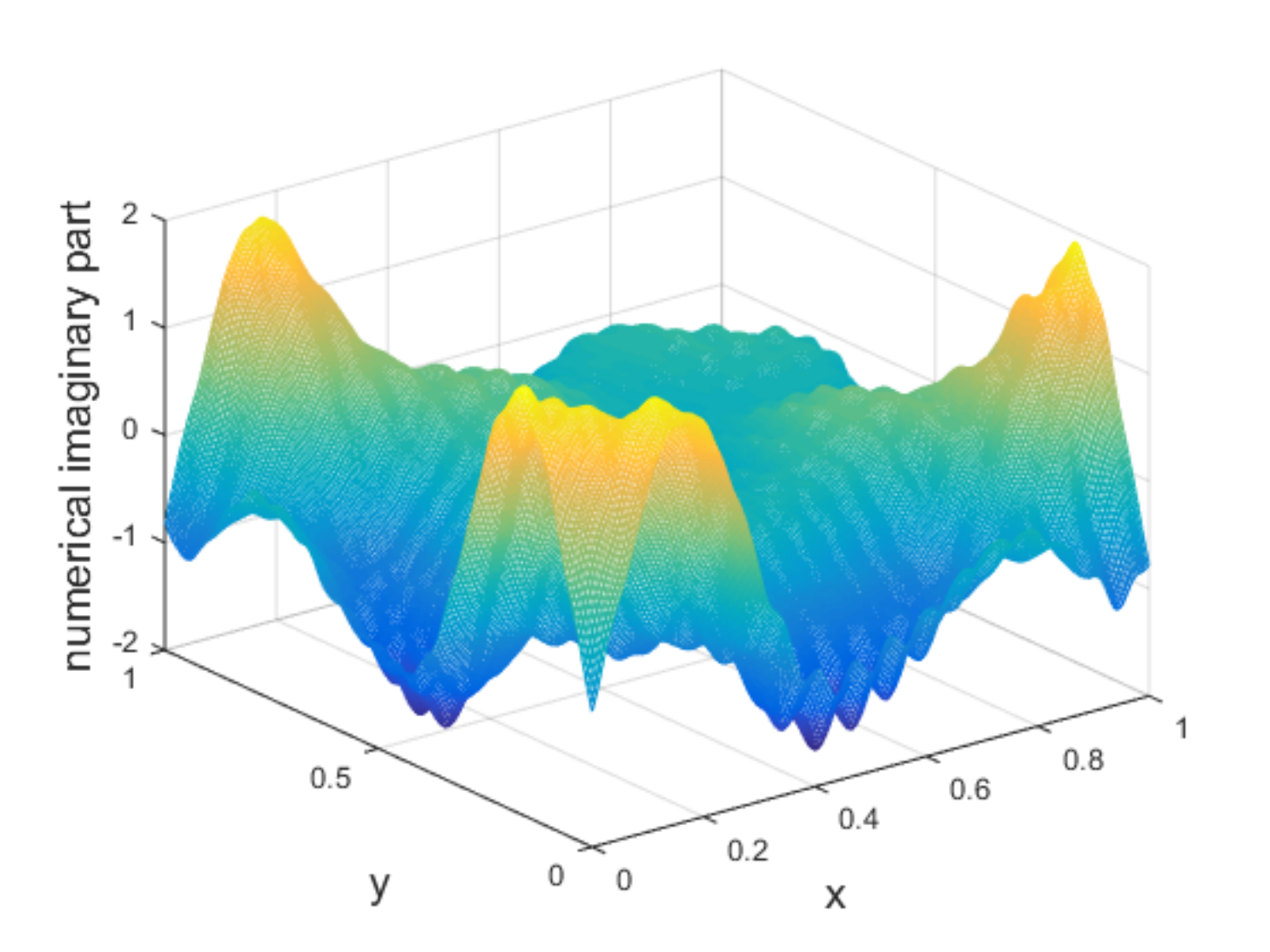}
				\label{fig:imagpart_1}
			\end{subfigure}
			\begin{subfigure}{0.39\textwidth}
				\centering
				\includegraphics[width=\textwidth]{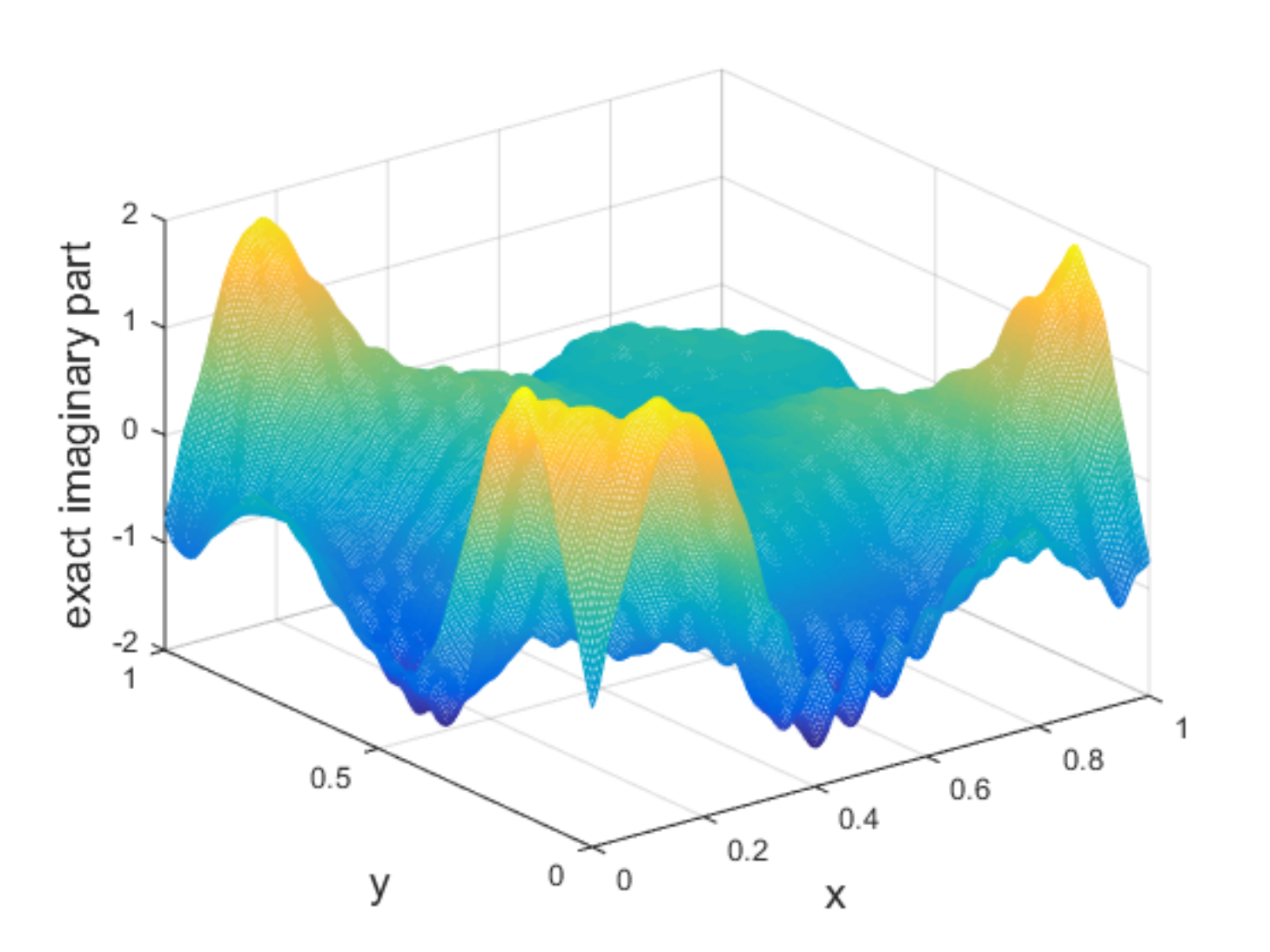}
				\label{fig:imagpart_2}
			\end{subfigure}
			\caption{Imaginary part of the wavefunction $ \psi^{\epsilon}(x,y,t) $ for \textbf{Example \ref{example4}} when $ H / \epsilon = 1/4$ and $ \epsilon=1/16 $. Left: numerical solution; Right: exact solution.}
			\label{fig8}
		\end{figure}
		
	\end{example}

	\begin{example}[2D case with a checkboard potential]\label{example5}
		The checkboard potential is of the following form
		\begin{equation}\label{eqn:potential5}
		v^{\epsilon}= u + \left\{
		\begin{aligned}
		(\cos(2\pi\frac{x}{\epsilon_2})+1)(\cos(2\pi\frac{y}{\epsilon_2})+1),\qquad&\{0\leq x, y\leq\frac{1}{2}\}\cup\{\frac{1}{2}\leq x, y\leq 1\},\\
		(\cos(2\pi\frac{x}{\epsilon_1})+1)(\cos(2\pi\frac{y}{\epsilon_1})+1),\qquad& \text{otherwise},\\
		\end{aligned}
		\right.
		\end{equation}
		where $\epsilon_1=1/8$, $\epsilon_2=1/16$, and the external potential $u(x,y)=|x-0.5|^2+|y-0.5|^2$. In the absence of the external potential,
		the profile of \eqref{eqn:potential5} is visualized in Figure \ref{fig:potential}. It allows for multiple spatial scales and
		discontinuities around interfaces, as in quantum metamaterials \cite{Quach:2011}.
		\begin{figure}[H]
			\centering
			\includegraphics[width=0.5\linewidth]{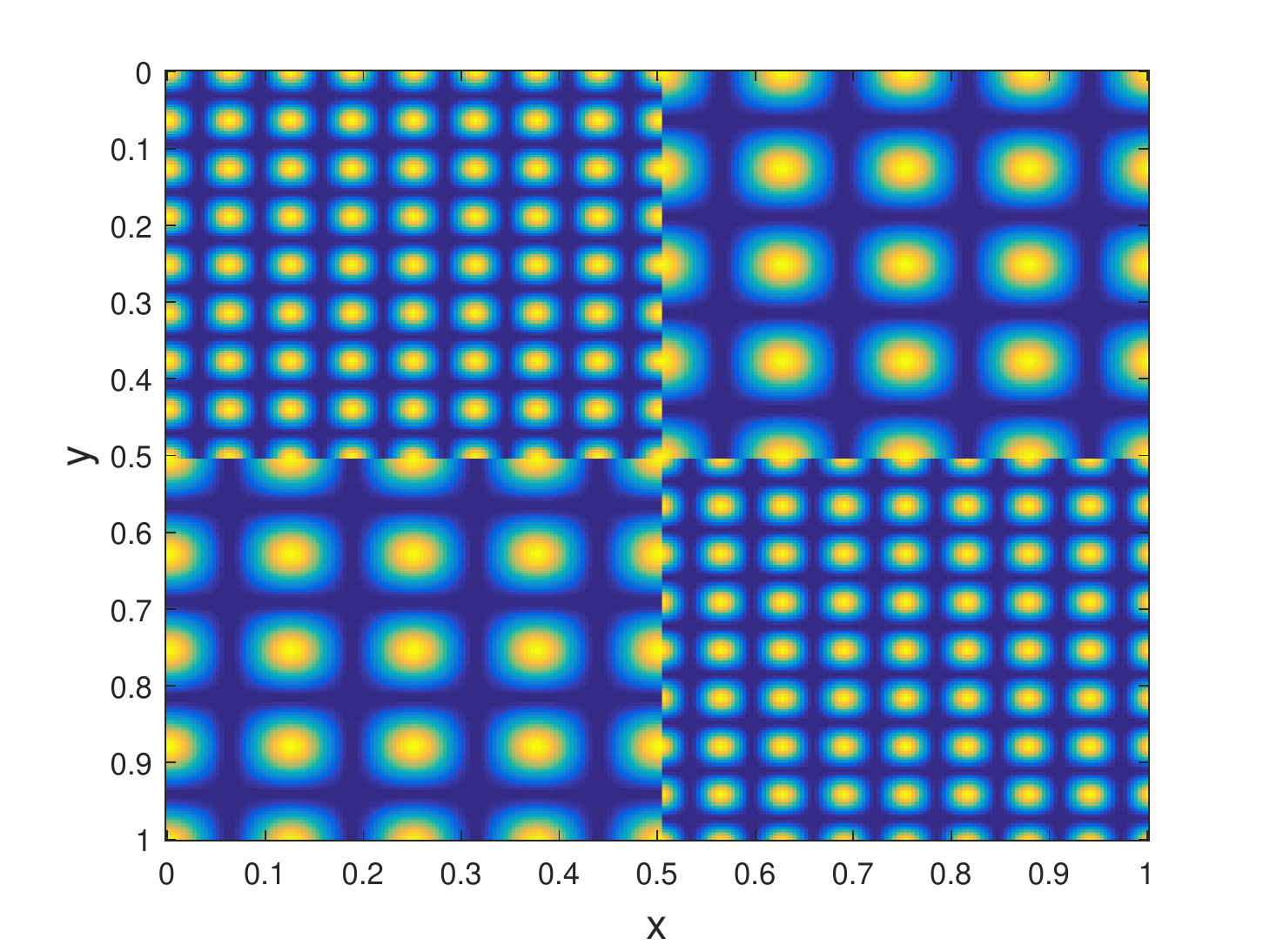}
			\caption{A checkboard-type potential defined over the unit square.}
			\label{fig:potential}
		\end{figure}
		
		Table \ref{table8} records the relative errors in both $L^2$ norm and $H^1$ norm for a series of coarse meshes satisfying
		$ H/\epsilon = 1/2, 1/4, 1/8$ with $\epsilon=1/16$. The meshsize condition \eqref{eqn:meshcondition} and convergence rates
		2 and 1 in $L^2$ norm and $H^1$ norm are suggested.
		\begin{table}[H]
			\centering
			\begin{tabular}{|r|r|r|r|r|}
				\hline
				$ H/\epsilon $ & $ \textrm{Error}_{L^2} $ & Order & $ \textrm{Error}_{H^1} $ & Order \\
				\hline
				$ 	1/2 $ &0.08832309&       &0.56796869&  \\
				$ 1/4  $& 0.01969196& 2.21  &0.18631339& 1.61\\
				$ 1/8 $& 0.00274243& 2.86  & 0.06010238& 1.63\\
				\hline
			\end{tabular}%
			\caption{Relative $L^2$ and $H^1$ errors of the wavefunction for \textbf{Example \ref{example5}} when $ \epsilon=1/16 $.}
			\label{table8}
		\end{table}
		
		Profiles of the position density function \eqref{eqn:density} and the energy density function \eqref{eqn:energy}
		when $\veps=1/16$ are plotted in Figure \ref{fig4} and Figure \ref{fig5}, respectively. Similar to those in previous exmaples,
		excellent agreements between numerical solutions and the exact solutions are observed again.
		\begin{figure}[H]
			\centering
			\begin{subfigure}{0.39\textwidth}
				\centering
				\includegraphics[width=\textwidth]{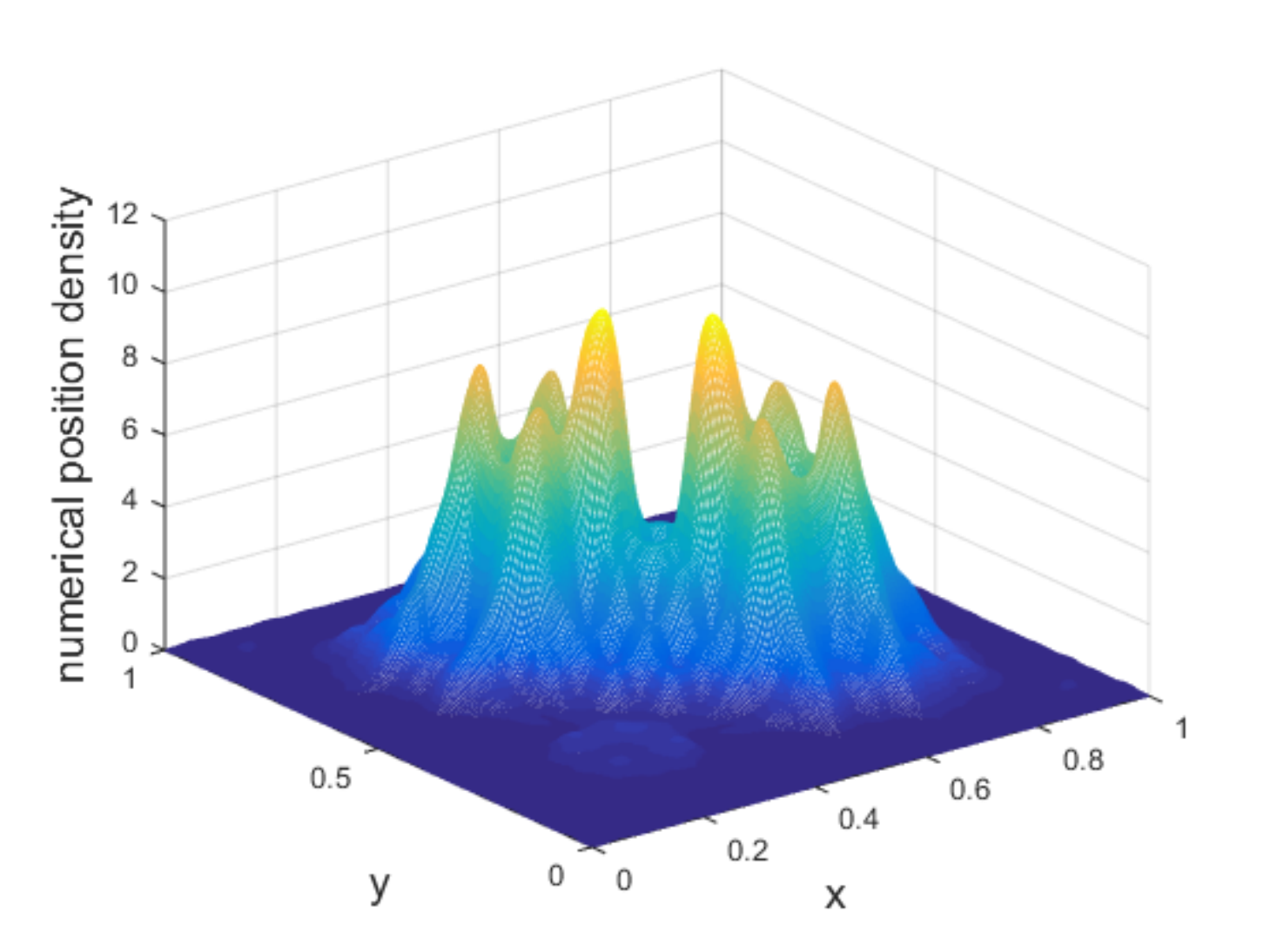}
				\label{fig:CHECKposition_1}
			\end{subfigure}
			\begin{subfigure}{0.39\textwidth}
				\centering
				\includegraphics[width=\textwidth]{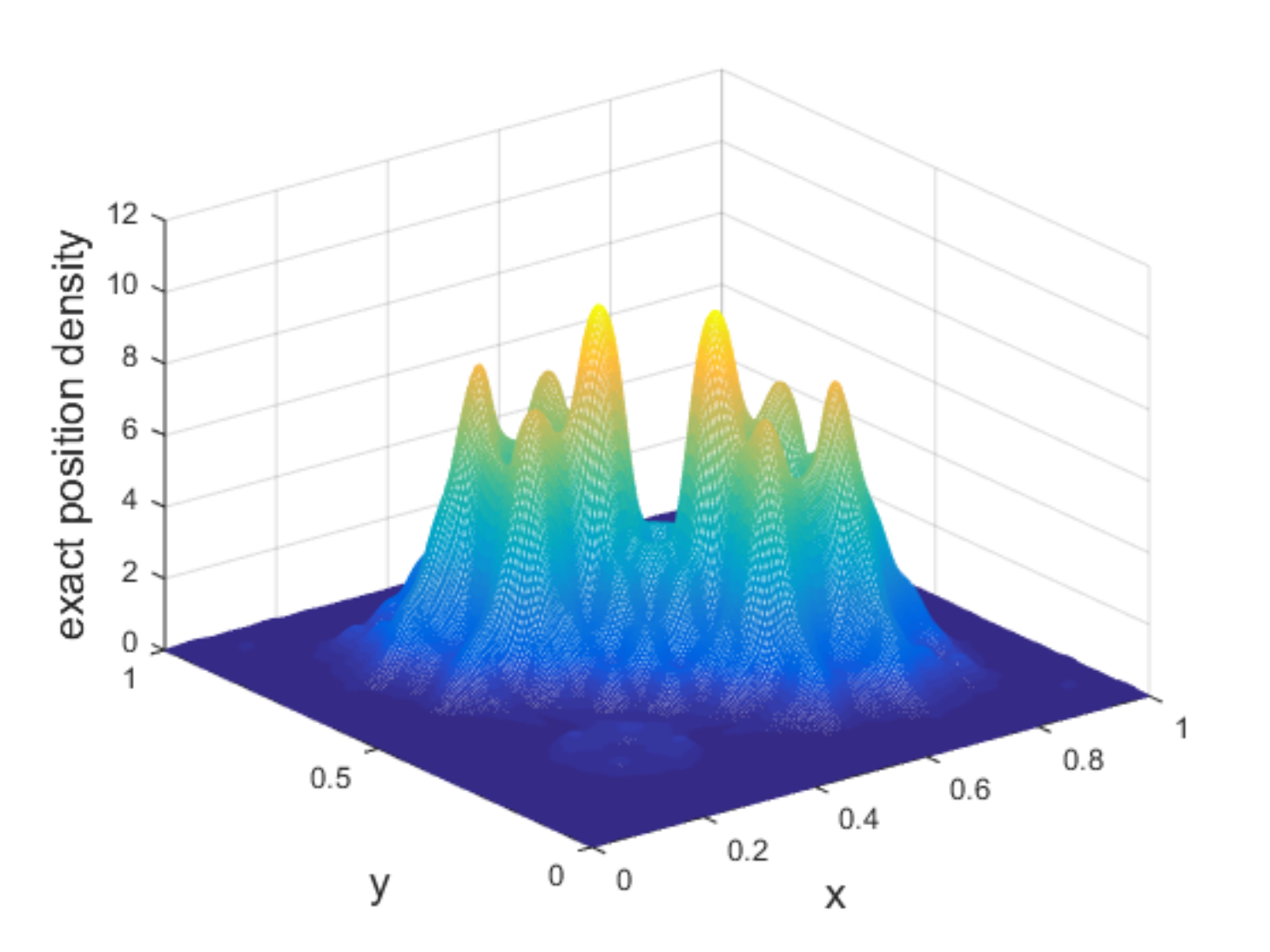}
				\label{fig:CHECKposition_2}
			\end{subfigure}
			\caption{Profiles of position density functions for \textbf{Example \ref{example5}} when $ H/\epsilon=1/4$ and $ \epsilon=1/16 $. Left: numerical solution; Right: exact solution.}
			\label{fig9}
		\end{figure}
		
		\begin{figure}[H]
			\centering
			\begin{subfigure}{0.39\textwidth}
				\centering
				\includegraphics[width=\textwidth]{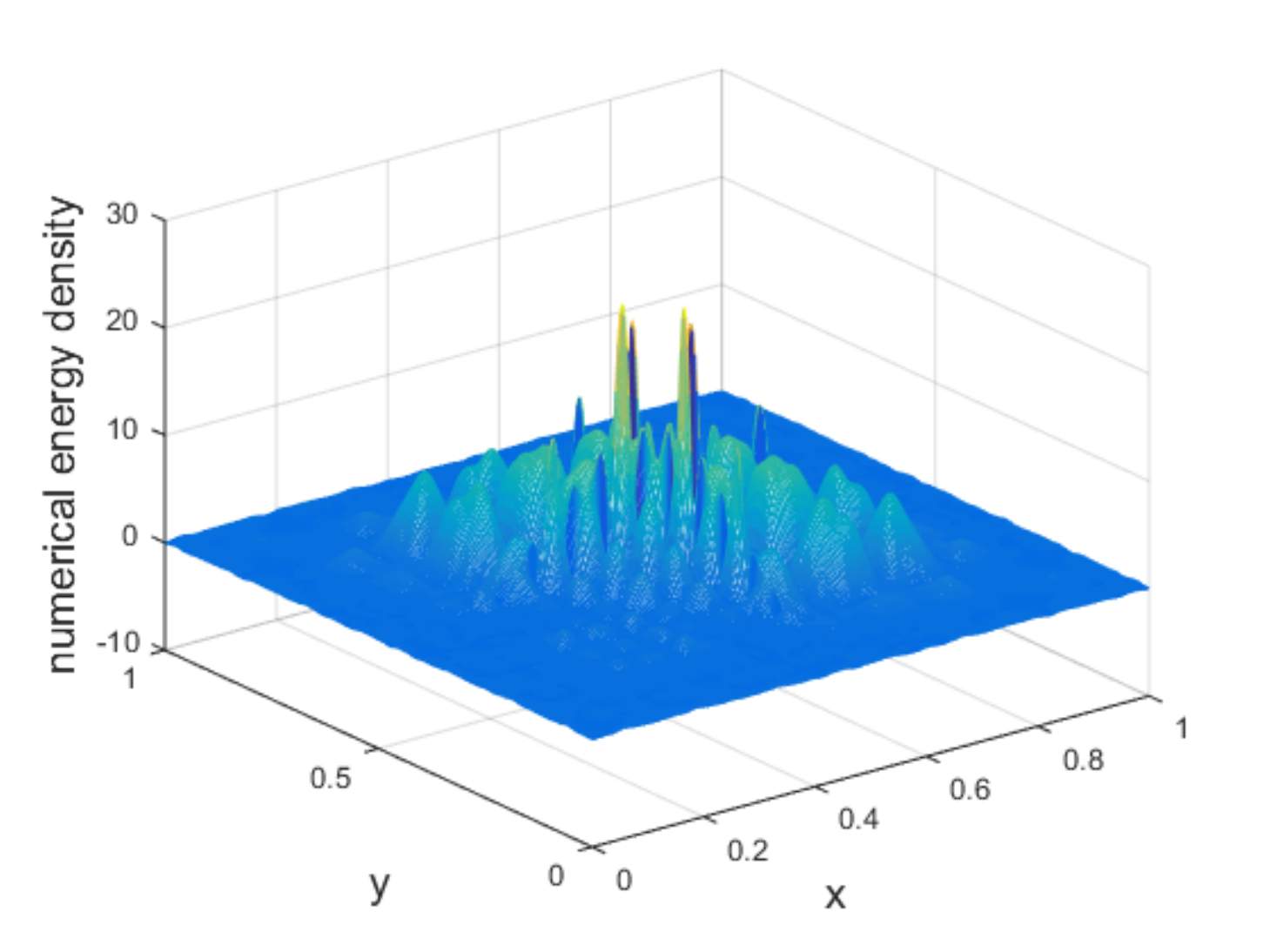}
				\label{fig:CHECKenergy_1}
			\end{subfigure}
			\begin{subfigure}{0.39\textwidth}
				\centering
				\includegraphics[width=\textwidth]{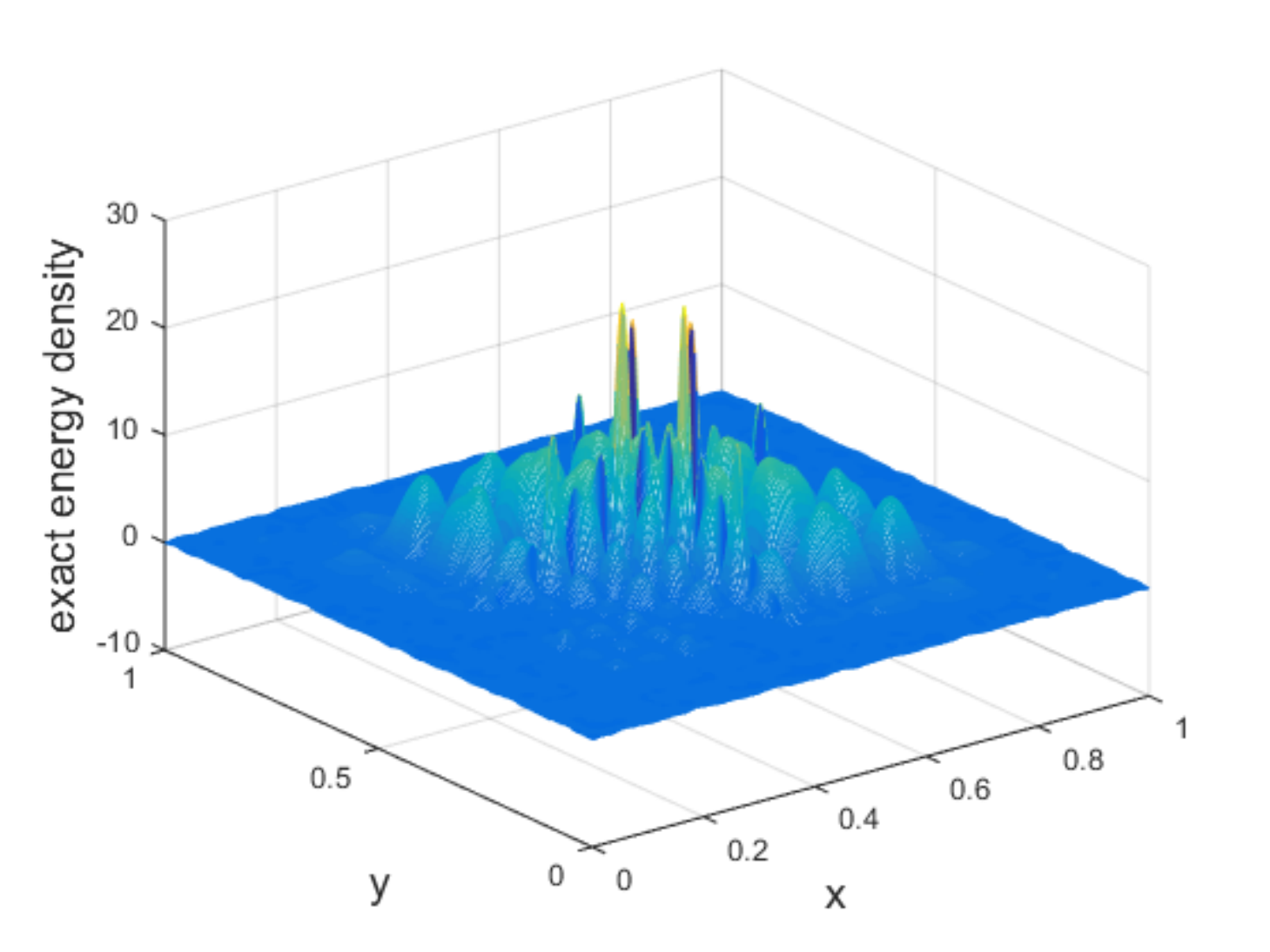}
				\label{fig:CHECKenergy_2}
			\end{subfigure}
			\caption{Profiles of energy density functions for \textbf{Example \ref{example5}} when $ H/\epsilon=1/4$ and $ \epsilon=1/16 $. Left: numerical solution; Right: exact solution.}
			\label{fig10}
		\end{figure}
		
	\end{example}

	\section{Conclusion and discussion} \label{sec:Conclusion}
	\noindent	
	In this paper, we have proposed a multiscale finite element method to solve the Schr\"{o}dinger equation with multiscale potentials in the semiclassical regime. The localized multiscale basis are constructed using sparse compression of the Hamiltonian operator, and thus are "blind" to the specific form of the potential. After an one-shot eigendecomposition, we can solve the resulting system of ordinary differential equations explicitly for the time evolution. In our approach, the spatial mesh size is $ H=\mathcal{O}(\epsilon) $ where $\epsilon$ is the semiclassical parameter and the time step $ k$ is independent of $\epsilon$. Numerous numerical examples in both 1D and 2D are given to demonstrate the efficiency and robustness of the proposed method.
	
	In the literature, asymptotics-based methods have the uniform $L^2-$approximation of the wavefunction, but do not have similar results in $H^1$ norm due to the ansatz used to construct the approximate solution. Our approach, however, have second-order and first-order rates of convergences in $L^2$ norm and $H^1$ norm, as illustrated in \secref{sec:NumericalExamples} by examples in both 1D and 2D with multiscale potentials. The convergence analysis of the proposed method will be presented in \cite{ChenMaZhang:prep}.
	
	From the perspective of physics, random information can be added to the Schr\"{o}dinger equation to study the Anderson
	localization phenomenon \cite{anderson1958absence}, which was studied in a recent work \cite{WuHuang:2016} that combines
	the Bloch decomposition-based split-step pseudospectral method and the generalized polynomial chaos method.
	Meanwhile, along another line, Hou, Ma, and Zhang proposed to build localized multiscale basis functions with
	the generalized polynomial chaos method \cite{hou2018model} to solve elliptic problems with random coefficients.
	One may expect a natural extension of \cite{hou2018model} would work for the random Schr\"{o}dinger equation.
	Unfortunately, it is not the case due to different natures of these two types of equations.
	Therefore, substantial work needs to be done to study the random Schr\"{o}dinger equation with multiple
	random inputs over long time.
	
	From the perspective of materials sciences, the proposed method can be combined with numerical methods for Landau-Lifshitz equation \cite{KruikProhl:2006, ChenLiuZhou:2016} to study current-driven domain wall dynamics \cite{Zutic:2004, Chen:2015} which are of great interest in spintronic devices.
	
	\section*{Acknowledgements}
	\noindent
	J. Chen acknowledges the financial support by National Natural Science Foundation of China via grant 21602149.	Z. Zhang acknowledges the financial support of Hong Kong RGC grants (Projects 27300616, 17300817, and 17300318) and National Natural Science Foundation of China via grant 11601457, Seed Funding Programme for Basic Research (HKU), and an RAE Improvement Fund from the Faculty of Science (HKU). Part of the work was done when J. Chen was visiting Department of Mathematics, University of Hong Kong. J. Chen would like to thank its hospitality.
	We would like to thank Professor Thomas Hou for stimulating discussions.
	
	
	\bibliographystyle{siam}
	\bibliography{ZWpaper}
	
	
\end{document}